\documentclass[12pt,reqno]{amsart}
\usepackage[T2A]{fontenc}
\usepackage[english]{babel}
\usepackage{amsaddr}
\usepackage{amsmath}
\usepackage{amssymb}
\usepackage{amsfonts}
\usepackage{epsfig}
\usepackage{srcltx}
\usepackage{subfigure}
\usepackage{float}
\usepackage{xcolor}
\usepackage{mathtools}
\usepackage[a4paper, mag=1000, includefoot, left=2cm, right=2cm, top=2cm, bottom=2cm, headsep=1cm, footskip=1cm]{geometry}
   \usepackage[unicode,colorlinks]{hyperref}
     \hypersetup{colorlinks=true, citecolor=blue, linkcolor=blue}

\newtheorem{Th}{Theorem}
\newtheorem{Lem}{Lemma}

\begin{document}

\thispagestyle{empty}

\title[]{Resonance in coupled nonlinear oscillators with decaying perturbations}

\author[O.A. Sultanov]{Oskar A. Sultanov}

\address{
Institute of Mathematics, Ufa Federal Research Centre, Russian Academy of Sciences, Chernyshevsky street, 112, Ufa 450008 Russia; \\
Bashkir State Pedagogical University n. a. M. Akmulla, Oktyabr'skoi Revolyutsii street, 3a, Ufa 450000 Russia.}
\email{oasultanov@gmail.com}


\maketitle

{\small
\begin{quote}
\noindent{\bf Abstract.} The influence of nonautonomous perturbations on a system of two coupled nonlinear non-identical oscillators is studied. Both the coupling intensity and the perturbation strength are assumed to decay in time according to a power law. We focus on resonance effects associated with the commensurability of the natural frequencies of the oscillators at some energy levels. In particular, we describe the conditions on perturbations and coupling parameters under which phase-locked solutions appear with the oscillator energies remaining asymptotically close to resonant levels. By combining the averaging method with the construction of Lyapunov functions, we derive a nonautonomous model system governing the perturbed dynamics and analyse the stability of resonant solutions in the phase-locking regime. The theoretical results are illustrated for a system of non-identical Duffing oscillators with decaying coupling.

\medskip



\end{quote}
}
{\small

\section*{Introduction}

Resonance plays a significant role in a wide range of applications, including mechanics, optics, engineering structures, and biological systems~\cite{BII88,SUZ88,AF06,LF09,RS16,MKSS18}. Such phenomena arise from the interaction between oscillators and external perturbations or from the coupling between subsystems. Depending on the system parameters and forcing conditions, they may lead either to an amplification of oscillations and a loss of stability or to the emergence of stable regimes. In the latter case, resonance can be used for efficient energy transfer, stabilization, synchronization, or the control of dynamics.

In nonlinear dynamics, the study of resonance phenomena is of great importance. In contrast to the linear case, where resonance is usually associated with the coincidence of natural and excitation frequencies, in nonlinear systems, the frequencies depend on the energy. In this case, resonant relations identify the energy levels where special asymptotic regimes may arise~\cite{BM61,BVC79,AKN06}. The analysis of such modes is essential for understanding the mechanisms responsible for the emergence of stable long-term dynamics.

In this paper, a system of two non-identical coupled nonlinear oscillators subject to nonautonomous perturbations is considered. Both the coupling intensity and the perturbations strength are assumed to decay in time. In this case, the considered system belongs to a class of asymptotically autonomous systems whose limiting dynamics is governed by the unperturbed oscillators. Note that perturbations with decaying intensity occur in a wide range of applications, including equations with time-dependent damping~\cite{JM23}, systems with vanishing control~\cite{AT18,OS20SAPM}, epidemic~\cite{CCT95} and cosmological~\cite{MLP26} models, and stellar~\cite{NY89} and DNA~\cite{SFR19} dynamics. Moreover, qualitative and asymptotic properties of such systems have been extensively studied. In particular, it is known that decaying additives may either preserve the limiting dynamics~\cite{LM56,LDP74} or lead to new stable states~\cite{HRT94,OS22Non,OS21IJBC}. The aim of this paper is to investigate whether such disturbances can induce resonant regimes with energies stabilized near resonant values and phases synchronized with a constant shift. Similar problems with small autonomous perturbations were considered in~\cite{KST09,LK18}, where synchronization and resonance capture were investigated. However, in this paper, the presence of a small parameter is not assumed, and the influence of decaying nonautonomous additives is discussed. 

Nonlinear resonance in planar oscillatory systems subject to decaying perturbations was previously studied in~\cite{OS25DCDS}, where it was shown that stable regimes with asymptotically constant amplitude may arise. Asymptotic analysis of the related linear and nonlinear problems was discussed in~\cite{BN10,OS23JMS}. In the case of two non-identical coupled oscillators, the situation becomes more complex. The perturbed nonautonomous system is four-dimensional with resonant dynamics determined by the joint evolution of two energies and two phases. Therefore, the methods developed for planar systems cannot be directly applied. The study of this case is the subject of the present paper. Let us also mention papers where related problems were considered. In particular, the asymptotics of solutions to systems of coupled linear oscillators with decaying coupling were constructed in~\cite{PNN13}, the persistence of invariant tori in Hamiltonian systems with rapidly decaying perturbations was discussed in~\cite{IP12,DS24}, and the stability of an equilibrium in a system of coupled weakly nonlinear oscillators with disturbances integrable on the semi-axis was analyzed in~\cite{MV23}. The present paper studies the existence and stability of resonant regimes in systems of coupled nonlinear oscillators with slowly decaying parametric perturbations.

The paper is organized as follows. Section~\ref{sec1} formulates the problem and gives motivating examples. The main results are presented in Section~\ref{sec2}, while the proofs are given in the subsequent sections. In particular, Section~\ref{sec3} is devoted to the construction of a change of variables that averages the leading asymptotic terms of the perturbed system. Section~\ref{sec4} contains the stability analysis of the corresponding truncated system. In Section~\ref{sec5}, a model system obtained from the averaged equations by dropping the remainder terms is studied, and solutions associated with an equilibrium of the truncated system are analyzed. In Section~\ref{sec6}, the persistence of these regimes in the full system is established by constructing suitable Lyapunov functions. In Section~\ref{sex}, the proposed theory is applied to non-identical Duffing oscillators with decaying coupling and perturbations. The paper concludes with a brief discussion of the results.

\section{Problem statement}\label{sec1}
Consider a system of four differential equations
\begin{gather}\label{PS}
\begin{split}
&\frac{dE_1}{dt}=f_1(E_1,E_2,\varphi_1,\varphi_2,t), 
\quad 
 \frac{d\varphi_1}{dt}-\omega_1(E_1) = g_1(E_1,E_2,\varphi_1,\varphi_2,t), \\
&\frac{dE_2}{dt}=f_2(E_1,E_2,\varphi_1,\varphi_2,t), 
\quad 
 \frac{d\varphi_2}{dt}-\omega_2(E_2) = g_2(E_1,E_2,\varphi_1,\varphi_2,t), 
\end{split}
\end{gather}
where $\omega_l(E_l)>0$, and the functions $f_l(E_1,E_2,\varphi_1,\varphi_2,t)$ and $g_l(E_1,E_2,\varphi_1,\varphi_2,t)$ are defined for all $E_1,E_2\in [0,\mathfrak E]$, with $\mathfrak E={\hbox{\rm const}}$, and $\varphi_1,\varphi_2\in\mathbb R$, $t\geq t_0\geq 1$. These functions are infinitely differentiable for $E_1,E_2\in (0,\mathfrak E]$ and are $2\pi$-periodic with respect to $\varphi_1$ and $\varphi_2$. The functions $f_i$ and $g_i$ play the role of perturbations of the autonomous system
\begin{gather}\label{US}
\frac{d\hat E_l}{dt}=0, \quad \frac{d\hat \varphi_l}{dt}=\omega_l(\hat E_l), \quad l\in\{1,2\},
\end{gather}
describing two independent non-isochronous oscillators with natural frequencies $\omega_1(\hat E_1)$ and $\omega_2(\hat E_2)$. The solutions $E_l(t)$ and $\varphi_l(t)$ of system \eqref{PS} correspond to the energy and the phase of the perturbed oscillations. 

It is assumed that the intensity of perturbations decays with time and the following asymptotic expansions hold
\begin{gather}\label{fgas}\begin{split}
&f_l(E_1,E_2,\varphi_1,\varphi_2,t)\sim  \sum_{k=1}^\infty t^{-k\alpha} f_{l,k}(E_1,E_2,\varphi_1,\varphi_2), \\ 
& g_l(E_1,E_2,\varphi_1,\varphi_2,t)\sim  \sum_{k=1}^\infty t^{-k\alpha} g_{l,k}(E_1,E_2,\varphi_1,\varphi_2)
\end{split}
\end{gather}
as $t\to\infty$, where the coefficients $f_{l,k}(E_1,E_2,\varphi_1,\varphi_2)$ and $g_{l,k}(E_1,E_2,\varphi_1,\varphi_2)$ are $2\pi$-periodic with respect to $\varphi_1$ and $\varphi_2$, and $\alpha\in (0,1]$. Note that the series in \eqref{fgas} are asymptotic as $t\to\infty$ such that for all $N> 1$ the following estimates hold: $f_l(E_1,E_2,\varphi_1,\varphi_2,t)-\sum_{k=1}^{N-1}t^{-k\alpha} f_{l,k}(E_1,E_2,\varphi_1,\varphi_2)=\mathcal O(t^{-N\alpha})$ and $g_l(E_1,E_2,\varphi_1,\varphi_2,t)-\sum_{k=1}^{N-1}t^{-k\alpha} g_{l,k}(E_1,E_2,\varphi_1,\varphi_2)=\mathcal O(t^{-N\alpha})$ as $t\to\infty$ uniformly for all $E_1,E_2\in [0,\mathfrak E)$ and $\varphi_1, \varphi_2\in\mathbb R$. Moreover, the expansions are assumed to be differentiable term-wise with respect to $E_i$, $\varphi_i$, and $t$, with the differentiated remainders satisfying analogous uniform estimates. In this case, the perturbed system \eqref{PS} is asymptotically autonomous, with its limiting system given by \eqref{US}. Thus, the functions $f_l$ and $g_l$ describe both the coupling between the oscillators and the parametric disturbances of their dynamics. In what follows, we refer to these functions collectively as perturbation terms. 

The goal of the paper is to study the effects of perturbations on the dynamics near the resonant values $A_1,A_2\in [0,\mathfrak E]$ such that 
\begin{gather}\label{rc}
	\varkappa\omega_1(A_1)=\kappa\omega_2(A_2), \quad \eta_i:=\omega_i'(A_i)\neq 0
\end{gather}
for some coprime integers $\varkappa,\kappa\in\mathbb Z\setminus\{0\}$. Consider the case where $A_1=0$ and $A:=A_2\neq 0$.
 
Let us specify the asymptotic behaviour of perturbations $f_1$ and $g_1$ in the vicinity of $E_1=A_1=0$:
\begin{gather}\label{as0}
f_1(E_1,E_2,\varphi_1,\varphi_2,t)=\mathcal O(E_1), \quad 
g_1(E_1,E_2,\varphi_1,\varphi_2,t)=\mathcal O(1)
\end{gather}
as $E_1\to 0$ uniformly for all $E_2\in [0,\mathfrak E]$, $\varphi_1,\varphi_2\in\mathbb R$, and $t\geq t_0$. Moreover, it is assumed that 
\begin{gather}\label{fgasik}
\begin{split}
	f_{l,k}(E_1,E_2,\varphi_1,\varphi_2)&\sim \sum_{i=0}^\infty E_1^{\frac{i}{2}}f_{l,k,i}(E_2,\varphi_1,\varphi_2), \\
	g_{l,k}(E_1,E_2,\varphi_1,\varphi_2)&\sim \sum_{i=0}^\infty E_1^{\frac{i}{2}}g_{l,k,i}(E_2,\varphi_1,\varphi_2),
\end{split}
\end{gather}
as $E_1\to 0$ uniformly for all $E_2\in [0,\mathfrak E]$, $\varphi_1$, $\varphi_2\in\mathbb R$, $l\in\{1,2\}$, where $f_{1,k,0}\equiv f_{1,k,1}\equiv 0$. 
Note that condition \eqref{as0} guarantees that the perturbations preserve the equilibrium of the first oscillator. Consider the example given by the following system of two coupled oscillators
\begin{gather}\label{Ex0}
\begin{split}
\frac{d^2x_1}{dt^2}+ U'_1(x_1) & = t^{-\alpha} \mathcal G_1\left(x_1,\frac{dx_1}{dt},x_2,\frac{dx_2}{dt}\right), \\
\frac{d^2x_2}{dt^2}+U'_2(x_2) & = t^{-\alpha} \mathcal G_2\left(x_1,\frac{dx_1}{dt},x_2,\frac{dx_2}{dt}\right),
\end{split}
\end{gather}
where $\alpha>0$, $U'_i(x)=w_i^2 x+u_i x^3+\mathcal O(x^5)$ as $x\to 0$, $w_i,u_i={\hbox{\rm const}}$, $w_i>0$, $\mathcal G_1(x_1,z_1,x_2,z_2)= x_1 \mathcal G_{1,1}(x_2,z_2)+z_1 \mathcal G_{1,2}(x_2,z_2)+\mathcal O(x_1^2+z_1^2)$ and $\mathcal G_2(x_1,z_1,x_2,z_2)=\mathcal O(1)$ as $\sqrt{x_1^2+z_1^2}\to 0$. Assume that there exist $\mathfrak E>0$ and $\mathfrak R>0$ such that for all $i\in\{1,2\}$ the level lines $\{(x,y)\in\mathbb R^2: 2U_i(x)+w_i^2 y^2=2E\}$ lying in the domain $\{(x,y): \sqrt{x^2+y^2}\leq \mathfrak R\}$ are closed curves for all $R\in [0,\mathfrak R]$ and do not contain any fixed points of the system
\begin{gather}\label{ExLim}
\frac{d\hat x}{dt}=w_i\hat y, \quad \frac{d\hat y}{dt}=-\frac{U_i'(\hat x)}{w_i},
\end{gather}
other than $(0,0)$. Let $x_i^-(E)<0<x_i^+(E)$ be the solutions of the equation $U_i(x)=E$ such that $U_i'(x_i^\pm(E))\neq 0$ with $E\in (0,\mathfrak E]$. Then, to each closed curve there correspond a periodic solution $x_i^0(t,E)$, $y_i^0(t,E)$ of \eqref{ExLim} with a period
\begin{gather*}
	T_i(E)=\int\limits_{x_i^-(E)}^{x_i^+(E)} \frac{\sqrt 2\,d\varsigma}{\sqrt{E-U_i(\varsigma)}}.
\end{gather*}
Define $\omega_i(E)\equiv 2\pi/T_i(E)$ and auxiliary $2\pi$-periodic functions 
\begin{gather*}
X_i(\varphi,E)\equiv \hat x_i^0 \left(\frac{\varphi}{\omega_i(E)},E\right), \quad
Y_i(\varphi,E)\equiv \hat y_i^0 \left(\frac{\varphi}{\omega_i(E)},E\right).
\end{gather*}
It can easily be checked that
\begin{gather*}
\omega_i(E)\partial_\varphi X_i=w_i Y_i, \quad \omega_i(E)\partial_\varphi Y_i=-
\frac{U_i'(X_i)}{w_i}, \\
2 U_i(X_i)+ w_i^2 Y_i^2 \equiv 2 E, \quad 
 {\hbox{\rm det}}\frac{\partial (X_i,Y_i)}{\partial (\varphi,E)}\equiv \begin{vmatrix} \partial_\varphi X_i & \partial_\varphi Y_i\\ \partial_E X_i & \partial_E Y_i\end{vmatrix} \equiv \frac{1}{w_i\omega_i(E)}.
\end{gather*}
Hence, the transformation $x_1=X_1(\varphi_1,E_1)$, $y_1=Y_1(\varphi_1,E_1)$, $x_2=X_2(\varphi_2,E_2)$, $y_2=Y_2(\varphi_2,E_2)$ is invertible for all $E_1,E_2\in (0,\mathfrak E]$ and $\varphi_1,\varphi_2\in[0,2\pi)$.
Thus, system \eqref{Ex0} in the variables $(E_1,E_2,\varphi_1,\varphi_2)$ takes the form \eqref{PS} with 
\begin{gather}\label{fgl}
\begin{split}
f_l(E_1,E_2,\varphi_1,\varphi_2,t) & \equiv t^{-\alpha}  w_l Y_l(\varphi_l,E_l)\mathcal G_l\left(X_1,w_1 Y_1,X_2,w_2 Y_2\right), \\
g_l(E_1,E_2,\varphi_1,\varphi_2,t) & \equiv -t^{-\alpha} 
 \omega_l(E_l) \partial_{E_l}X_l(\varphi_l,E_l)\mathcal  G_l\left(X_1,w_1 Y_1,X_2,w_2 Y_2\right).
\end{split}
\end{gather}
Since $X_1(\varphi_1,E_1)=\sqrt{2E_1/w_1^2}\cos\varphi_1+\mathcal O(E_1)$, $Y_1(\varphi_1,E_1)=-\sqrt{2E_1/w_1^2}\sin\varphi_1+\mathcal O(E_1)$, $\omega_1(E_1)=w_1+\mathcal O(E_1)$ as $E_1\to 0$, we see that 
\begin{align*}
f_1& =-t^{-\alpha} \frac{2 E_1 \sin \varphi_1}{w_1} \left(\mathcal G_{1,1}(X_2,w_2 Y_2)\cos\varphi_1-\mathcal G_{1,2}(X_2,w_2 Y_2)\sin\varphi_1+\mathcal O(\sqrt{E_1})\right), \\
g_1& =-t^{-\alpha}\frac{\cos\varphi_1 }{w_1}\left(\mathcal G_{1,1}(X_2,w_2 Y_2)\cos\varphi_1-  \mathcal G_{1,2}(X_2,w_2 Y_2)\sin\varphi_1+\mathcal O(\sqrt{E_1})\right).
\end{align*} 
Thus, we obtain \eqref{as0}.

Consider the model example with $\omega_1(E)\equiv 1/2+E$, $\omega_2(E)\equiv 1-E$,
\begin{gather}
\label{FGmod}
\begin{split}
& f_i(E_1,E_2,\varphi_1,\varphi_2,t)\equiv t^{- \alpha }E_i \sin\varphi_i (a_{i,1} \sin\varphi_1+ a_{i,2} \sin\varphi_2), \\ 
& g_i(E_1,E_2,\varphi_1,\varphi_2,t)\equiv t^{- \alpha } \cos\varphi_i (b_{i,1}\sin \varphi_1+ b_{i,2}\sin\varphi_2)
\end{split}
\end{gather}
and constant parameters $a_{i,j},b_{i,j}\in\mathbb R$. It is clear that $f_i$ and $g_i$ satisfy \eqref{fgas}, \eqref{as0}, and the resonant condition \eqref{rc} holds with $\varkappa=\kappa=1$, $A=1/2$, $\eta_1=1$ and $\eta_2=-1$. If $f_i\equiv g_i=0$, then $E_i(t)=E_{i,0}$, $\varphi_i(t)=\omega_i(E_{i,0}t)+\varphi_{i,0}$ with arbitrary constants $E_{i,0}$ and $\varphi_{i,0}$. Under certain conditions on the parameters, this behaviour can persist in the perturbed system (see Fig.~\ref{Exm}, a), or growing and decaying solutions may occur (see Fig.~\ref{Exm}, b). In addition, attracting solutions with asymptotically constant resonant energies may arise (see Fig.~\ref{Exm}, c). This paper discusses the conditions that guarantee the existence and stability of such resonant solutions in system \eqref{PS} with perturbations satisfying \eqref{fgas} and \eqref{as0}. 
\begin{figure}
\centering
\subfigure[ ]{
\includegraphics[width=0.3\linewidth]{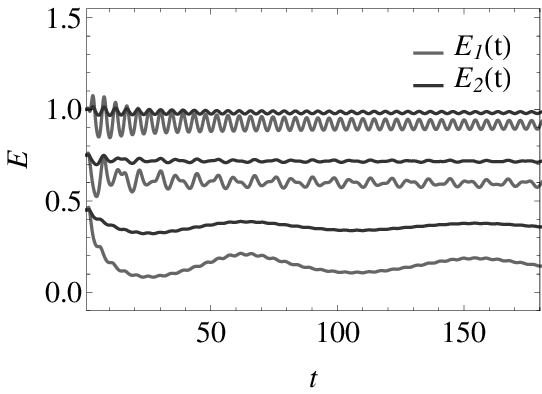}
}
\hspace{1ex}
 \subfigure[ ]{
 \includegraphics[width=0.3\linewidth]{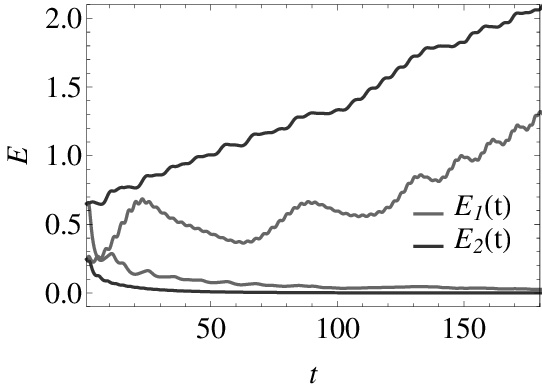}
}
\hspace{1ex}
\subfigure[ ]{
 \includegraphics[width=0.3\linewidth]{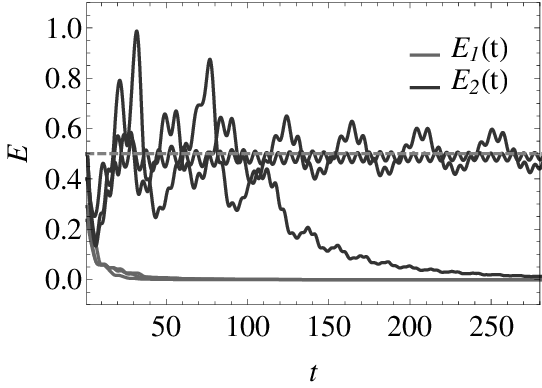}
}
\caption{\small The evolution of $E_1(t)$ and $E_2(t)$ for solutions of system \eqref{Exm} with $f_i$ and $g_i$, defined by \eqref{FGmod}, $\alpha=1/2$, different values of the parameters $a_{i,j}$ and $b_{i,j}$, and different initial conditions. The dashed curve corresponds to $E(t)\equiv 1/2$.} \label{Exm}
\end{figure}

\section{Main results}\label{sec2}

Let $A\in \mathbb R$, $A\in (0,\mathfrak E)$ be a parameter satisfying resonance condition \eqref{rc}.
Define the domain
\begin{align*}
&\mathfrak D_{\epsilon,\tau}:=\{(\rho,v,\psi,\varphi)\in\mathbb R^4: \quad 0 \leq \rho  \leq (1-\epsilon)\tau^{\frac{\alpha}{2}} \sqrt{\mathfrak E} , \quad  \epsilon \leq v+ \tau^{\frac{\alpha}{2}} A \leq \tau^{\frac{\alpha}{2}} \mathfrak E -\epsilon\}
\end{align*} 
with some $\epsilon\in [0, \epsilon_\ast)$, $\epsilon_\ast:=\min\{1,A,\mathfrak E-A\})$ and $\tau\geq t_0\geq 1$. Let the angle brackets $\langle F(\varphi) \rangle_{\varkappa \varphi}$ denote the average value of a function $F(\varphi)$ over an interval $[0,2\pi\varkappa]$,
\begin{gather*}
\langle F(\varphi)\rangle_{\varkappa \varphi} \equiv \frac{1}{2\pi\varkappa}\int\limits_0^{2\pi\varkappa} F(\varphi)\,d\varphi.
\end{gather*}
Then, we have the following

\begin{Th}\label{Th1}
Let system \eqref{PS} satisfy \eqref{fgas}, \eqref{rc} and \eqref{fgasik}. Then, for all $1< N\leq 2\alpha^{-1}$ and $\epsilon\in(0,\epsilon_\ast)$ there exist $t_\ast\geq t_0$ and the transformations $(E_1,E_2,\varphi_1,\varphi_2)\mapsto (\mathcal R,\mathcal E,\mathcal S,\varphi)\mapsto (\rho,v,\psi,\varphi)$,
\begin{gather}
\label{ch1}
	E_1=t^{-\alpha}\mathcal R^2, \quad 
	E_2=A+t^{-\frac{\alpha}{2}}\mathcal E, \quad 
	\varphi_1=\mathcal S+\frac{\kappa}{\varkappa}\varphi, \quad 
	\varphi_2=\varphi, \\
\label{ch2}
	\mathcal R=\rho+ P(\rho,v,\psi,\varphi,t), \quad 
	\mathcal E=v+ Q(\rho,v,\psi,\varphi,t), \quad 
	\mathcal S =\psi + Z(\rho,v,\psi,\varphi,t),
\end{gather}
where $ P(\rho,v,\psi,\varphi,t)$, $ Q(\rho,v,\psi,\varphi,t)$, $ Z(\rho,v,\psi,\varphi,t)$ satisfy the inequalities
\begin{gather}\label{tildeest}
| P(\rho,v,\psi,\varphi,t)|\leq \epsilon \rho, \quad 
| Q(\rho,v,\psi,\varphi,t)|\leq \epsilon, \quad 
| Z(\rho,v,\psi,\varphi,t)|\leq \epsilon
\end{gather}
for all $(\rho,v,\psi,\varphi)\in\mathfrak D_{\epsilon,t_\ast}$ and $t\geq t_\ast$,
such that for all $E_1,E_2\in [0,\mathfrak E] $, $\varphi_1,\varphi_2\in\mathbb R$ system \eqref{PS} can be transformed into
\begin{gather}\label{rvpp}
\begin{split}
\frac{d\rho}{dt}&=\Pi(\rho,v,\psi,t)+ \tilde\Pi(\rho,v,\psi,\varphi,t), \\
\frac{dv}{dt}&=\Lambda(\rho,v,\psi,t)+ \tilde\Lambda(\rho,v,\psi,\varphi,t), \\
\frac{d\psi}{dt}&=\Omega(\rho,v,\psi,t)+ \tilde\Omega(\rho,v,\psi,\varphi,t),\\
 \frac{d\varphi}{dt}&=\omega_2(A)+\tilde {\mathcal G}(\rho,v,\psi,\varphi,t),
\end{split}
\end{gather}
with
\begin{gather}
\label{LLO}
\begin{split}
&\Pi(\rho,v,\psi,t)\equiv \sum_{K=2}^N t^{-\frac{K\alpha}{2}}\Pi_{K}(\rho,v,\psi), \quad \Lambda(\rho,v,\psi,t)\equiv \sum_{K=1}^N t^{-\frac{K\alpha}{2}}\Lambda_{K}(\rho,v,\psi), \\
&\Omega(\rho,v,\psi,t)\equiv \sum_{K=1}^N t^{-\frac{K\alpha}{2}}\Omega_K(\rho,v,\psi),
\end{split}
\end{gather}
where $\Pi_{K}(\rho,v,\psi)$, $\Lambda_{K}(\rho,v,\psi)$ and  $\Omega_K(\rho,v,\psi)$ are $2\pi$-periodic in $\psi$, $\Pi_{K}(\rho,v,\psi)=\mathcal O(\rho)$, $\Lambda_{K}(\rho,v,\psi)=\mathcal O(1)$, $\Omega_{K}(\rho,v,\psi)=\mathcal O(1)$ as $\rho\to 0 $ uniformly for all $(0,v,\psi,0)\in\mathfrak D_{\epsilon,t_\ast}$. In particular, 
\begin{gather*}
\Pi_2(\rho,v,\psi)\equiv  \rho \pi(\psi), \quad  
\Lambda_1(\rho,v,\psi)\equiv   \lambda(\psi), \quad 
 \Omega_1(\rho,v,\psi)\equiv  \nu v, \\
\Lambda_2(\rho,v,\psi)\equiv   \rho \lambda_{2,1}(\psi)+v \lambda_{2,2}(\psi), \quad   
\Omega_2(\rho,v,\psi)\equiv   \nu_{2,0}(\psi)+\rho^2 \nu_{2,1}+v^2 \nu_{2,2},
\end{gather*}
where
\begin{gather*}
\pi(\psi)\equiv \frac{1}{2} \left\langle f_{1,1,2}\left(A,\psi+\frac\kappa \varkappa \varphi,\varphi\right)+\delta_{1,\alpha}\right\rangle_{\varkappa \varphi},\quad 
\lambda(\psi)\equiv \left\langle f_{2,1,0}\left(A,\psi+\frac\kappa \varkappa \varphi,\varphi\right)\right\rangle_{\varkappa \varphi}, \quad 
\nu =  -\frac{\kappa  \eta_2}{\varkappa}, 
\\
\lambda_{2,1}(\psi)\equiv \left\langle f_{2,1,1}\left(A,\psi+\frac\kappa \varkappa \varphi,\varphi\right)\right\rangle_{\varkappa \varphi}, \quad \lambda_{2,2}(\psi)= \left\langle \partial_{E_2} f_{2,1,0}\left(A,\psi+\frac\kappa \varkappa \varphi,\varphi\right)\right\rangle_{\varkappa \varphi}+ \frac{\delta_{1,\alpha}}{2}, \\
\nu_{2,0}(\psi)\equiv   
\left\langle g_{1,1,0}\left(A,\psi+\frac\kappa \varkappa \varphi,\varphi\right)-\frac{\kappa}{\varkappa}g_{2,1,0}\left(A,\psi+\frac\kappa \varkappa \varphi,\varphi\right)\right\rangle_{\varkappa \varphi}, \quad 
\nu_{2,1}=\mathcal \eta_1, \quad 
\nu_{2,2}=-\frac{\kappa \omega_2''(A)}{2\varkappa}.
\end{gather*}
The remainders satisfy the following estimates
\begin{gather}\label{tildeLLOP}
\begin{split}
&\tilde\Pi(\rho,v,\psi,\varphi,t)=\mathcal O\left(t^{-\frac{(N+1)\alpha}{2}}\right), \quad
\tilde\Lambda(\rho,v,\psi,\varphi,t)=\mathcal O\left(t^{-\frac{(N+1)\alpha}{2}}\right), \\
& \tilde\Omega(\rho,v,\psi,\varphi,t)=\mathcal O\left(t^{-\frac{(N+1)\alpha}{2}}\right), \quad 
\tilde{\mathcal G}(\rho,v,\psi,\varphi,t)=\mathcal O\left(t^{-\frac{\alpha}{2}}\right)
\end{split}
\end{gather}
as $t\to\infty$ uniformly for all $(\rho,v,\psi,\varphi)\in\mathfrak D_{\epsilon,t_\ast}$.
\end{Th}
The proof of Theorem~\ref{Th1} is contained in Section~\ref{sec3}.

Note that if $(\rho(t),v(t),\psi(t),\varphi(t))\in\mathfrak D_{\epsilon,t_\ast}$ for all $t\geq t_\ast$, then, by \eqref{ch1}, \eqref{ch2} and \eqref{tildeest}, it follows that $E_1(t),E_2(t)\in[0,\mathfrak E]$ as $t\geq t_\ast$. Moreover, from the last equation in \eqref{rvpp} it follows that $|\varphi(t)|\to \infty$ as $t\to \infty$ while $(\rho(t),v(t),\psi(t),0)\in\mathfrak D_{\epsilon,t_\ast}$. Since the right-hand sides of equations for $\rho(t)$, $v(t)$, and $\psi(t)$ depend periodically on $\varphi$ only through rapidly decaying terms, the behavior of $\varphi(t)$ does not have a significant effect on the dynamics of the solutions, and the last equation in \eqref{rvpp} can be neglected. 

Let $N\in [2,2\alpha^{-1}]$. Consider the model system
\begin{gather}\label{rvpmod}
\frac{d\hat\rho}{dt}=\Pi(\hat\rho,\hat v,\hat\psi,t), \quad
\frac{d\hat v}{dt}=\Lambda(\hat\rho,\hat v,\hat\psi,t), \quad
\frac{d\hat\psi}{dt}=\Omega(\hat\rho,\hat v,\hat\psi,t),
\end{gather}
obtained from \eqref{rvpp} by dropping the remainders $\tilde \Pi$, $\tilde \Lambda$, $\tilde \Omega$ and the equation for $\varphi$. It can easily be checked that the truncated system 
\begin{gather}\label{rvpls}
\frac{d\hat\rho}{dt}=t^{-\alpha} \hat \rho \pi(\hat\psi), \quad
\frac{d\hat v}{dt}=t^{-\frac{\alpha}{2}}\lambda(\hat\psi), \quad
\frac{d\hat\psi}{dt}=t^{-\frac{\alpha}{2}}\nu \hat v
\end{gather}
has a stationary solution if and only if the equation $\lambda(\psi)=0$ has a root. 
Assume that
\begin{gather}\label{asl}
\exists\, \psi_0\in\mathbb R: \quad \lambda(\psi_0)=0, \quad \lambda'(\psi_0)\neq 0.
\end{gather}
Then, system \eqref{rvpls} has a solution $\hat\rho(t)\equiv \hat v(t)\equiv 0$, $\hat\psi(t)\equiv \psi_0$. Moreover, we have the following lemma.

\begin{Lem}\label{Lem0}
Let assumption \eqref{asl} hold. 
\begin{itemize}
	\item If $\pi(\psi_0)>0$ or $\lambda'(\psi_0)\nu >0$, the solution $(0,0,\psi_0)$ of system \eqref{rvpls} is unstable. 
	\item If $\pi(\psi_0)<0$ and $\lambda'(\psi_0)\nu <0$, the solution $(0,0,\psi_0)$ of system \eqref{rvpls} is stable.
\end{itemize}
\end{Lem}

The proof is contained in Section~\ref{sec4}.

Note that in the case of stability of the equilibrium $(0,0,\psi_0)$ in the truncated system \eqref{rvpls}, similar dynamics is observed in the model system. Define $\vartheta(\psi_0):=\lambda_{2,2}(\psi_0)+\nu'_{2,0}(\psi_0)$. Then we have

\begin{Lem}\label{Lem1}
Let assumption \eqref{asl} hold, $\pi(\psi_0)<0$ and $\lambda'(\psi_0)\nu<0$. 
\begin{itemize}
	\item If $\vartheta(\psi_0)<0$, then system \eqref{rvpmod} has an asymptotically stable particular solution $(\hat\rho_0(t),\hat v_0(t),\hat\psi_0(t))$ such that 
\begin{itemize}
\item $\hat\rho_0(t)=\mathcal O(t^{-\frac\alpha 2})$, $\hat v_0(t)=\mathcal O(t^{-\frac \alpha 2})$, $\hat\psi_0(t)=\psi_0+\mathcal O(t^{-\alpha})$ as $t\to\infty$ if $0<\alpha<1$,
\item $\hat\rho_0(t)=\mathcal O(t^{-\chi})$, $\hat v_0(t)=\mathcal O(t^{-\chi})$, $\hat\psi_0(t)=\psi_0+\mathcal O(t^{-\chi})$ as $t\to\infty$ if $\alpha=1$,
\end{itemize}
with some $\chi\in(0,\chi_0)$, $\chi_0=\min\{1,2|\pi(\psi_0)|,|\vartheta(\psi_0)|\}/2$.
\item If $\vartheta(\psi_0)>0$, then there exists $\varepsilon>0$ such that for all $\delta\in (0,\varepsilon)$ there is $\tau_\ast\geq t_\ast$ such that  the solution 
$(\hat\rho(t),\hat v(t),\hat\psi(t))$ of system \eqref{rvpmod} with initial data 
$|\hat\rho(\tau_s)|+|\hat v(\tau_s)|+|\hat \psi(\tau_s)-\psi_0|\leq \delta$ with some $\tau_s\geq \tau_\ast$ satisfies the inequality
$|\hat\rho(\tau_e)|+|\hat v(\tau_e)|+|\hat\psi(\tau_e)-\psi_0|\geq \varepsilon$ at some $\tau_e>\tau_s$.
\end{itemize}
\end{Lem}

The proof is contained in Section~\ref{sec5}.

It follows from Lemma~\ref{Lem1} that a stable equilibrium $(0,0,\psi_0)$ of the truncated system generates a stable solution $(\hat\rho_0(t), \hat v_0(t), \hat\psi_0(t))$ of the model system such that $\hat\rho_0(t)\to 0$, $\hat v_0(t)\to 0$, and $\hat\psi_0(t)\to \psi_0$ as $t\to\infty$. Let us show that such a solution corresponds to a phase synchronization regime in system \eqref{PS}.

\begin{Th}\label{Th2}
Let system \eqref{PS} satisfy \eqref{fgas}, \eqref{rc}, \eqref{fgasik}, and assumption \eqref{asl} hold.
\begin{itemize}
\item If $\pi(\psi_0)<0$, $\lambda'(\psi_0)\nu<0$ and $\vartheta(\psi_0)<0$, then for all $\varepsilon>0$ there exist $\delta_\varepsilon>0$ and $t_\varepsilon>0$ such that for all $t_s\geq t_\varepsilon$ any solution $(E_1(t),E_2(t),\varphi_1(t),\varphi_2(t))$ of system \eqref{PS} with initial data $|E_1(t_s)|+|E_2(t_s)-A|+|\varphi_1(t_s)-\kappa\varphi_2(t_s)/\varkappa-\psi_0|\leq \delta_\varepsilon$ satisfies the inequality
\begin{gather*}
|E_1(t)|+|E_2(t)-A|+\left|\varphi_1(t)-\frac{\kappa}{\varkappa}\varphi_2(t)-\psi_0\right|\leq \varepsilon
\end{gather*}
for all $t> t_s$.
\item If one of the following conditions holds:
(i) $\pi(\psi_0)>0$, (ii) $\lambda'(\psi_0)\nu>0$, (iii) $\pi(\psi_0)<0$, $\lambda'(\psi_0)\nu<0$, $\vartheta(\psi_0)>0$, then
there exists $\varepsilon>0$ such that for all $\delta\in (0,\varepsilon)$ there is $\tau_\ast\geq t_\ast$ such that the solution $(E_1(t),E_2(t),\varphi_1(t),\varphi_2(t))$ of system \eqref{PS} with initial data 
$|E_1(\tau_s)|+|E_2(\tau_s)-A|+|\varphi_1(\tau_s)-\kappa\varphi_2(\tau_s)/\varkappa-\psi_0|\leq \delta$ with some $\tau_s\geq \tau_\ast$ satisfies the inequality
\begin{gather*}
|E_1(\tau_e)|+|E_2(\tau_e)-A|+\left|\varphi_1(\tau_e)-\frac{\kappa}{\varkappa}\varphi_2(\tau_e)-\psi_0\right|\geq \varepsilon
\end{gather*}
at some $\tau_e>\tau_s$.

\end{itemize}
\end{Th}

Finally, consider the case when instead of \eqref{asl} the following assumption holds
\begin{gather}\label{asl2}
\lambda(\psi)\neq 0 \quad \forall\,\psi\in\mathbb R.
\end{gather}
Then, we have the following
\begin{Th}\label{Th3}
Let assumption \eqref{asl2} hold. Then, there exists $t_s\geq t_\ast$ such that $|v(t)|$ and $|\psi(t)|$ for solutions of system \eqref{rvpp} with initial data $(\rho(t_s),v(t_s),\psi(t_s),\varphi(t_s))\in\mathfrak D_{\epsilon,t_\ast}$ increase as $t > t_s$ until they reach the boundary of the domain $\mathfrak D_{\epsilon,t_\ast}$ in a finite time.
\end{Th}

In this case, the phase difference $\varphi_1(t)-\kappa \varphi_2(t)/\varkappa \to \infty$ as $t\to\infty$ and the resonant solutions with $r_1(t)\approx 0$ and $r_2(t)\approx A$ do not occur.

The proofs of Theorems~\ref{Th2} and \ref{Th3} are contained in Section~\ref{sec6}.

\section{Change of variables} 
\label{sec3}
\begin{proof}[Proof of Theorem~\ref{Th1}]
Substituting \eqref{ch1} into \eqref{PS} yields
\begin{gather}\label{EESp}\begin{split}
\frac{d\mathcal R}{dt}=&\,\mathcal F_1(\mathcal R,\mathcal E,\mathcal S,\varphi,t), \\
\frac{d\mathcal E}{dt}=&\,\mathcal F_2(\mathcal R,\mathcal E,\mathcal S,\varphi,t), \\
\frac{d\mathcal S}{dt}=&\,\mathcal G_1(\mathcal R,\mathcal E,\mathcal S,\varphi,t), \\
\frac{d\varphi}{dt}=&\,\omega_2(A)+\mathcal G_2(\mathcal R,\mathcal E,\mathcal S,\varphi,t)
\end{split}
\end{gather}
where 
\begin{align*}
\mathcal F_1(\mathcal R,\mathcal E,\mathcal S,\varphi,t)\equiv 
& \, 
\frac{t^{\alpha}}{2\mathcal R}f_1\left(t^{-\alpha}\mathcal R^2, A+t^{-\frac{\alpha}{2}}\mathcal E, \mathcal S+\frac{\kappa}{\varkappa}\varphi,\varphi,t\right)+t^{-1}\frac{\alpha \mathcal R}{2}, \\
\mathcal F_2(\mathcal R,\mathcal E,\mathcal S,\varphi,t)\equiv 
& \, 
t^{\frac{\alpha}{q}}f_2\left(t^{-\alpha}\mathcal R^2, A+t^{-\frac{\alpha}{2}}\mathcal E, \mathcal S+\frac{\kappa}{\varkappa}\varphi,\varphi,t\right)+t^{-1}\frac{\alpha \mathcal E}{2}, \\
\mathcal G_1(\mathcal R,\mathcal E,\mathcal S,\varphi,t)\equiv 
& \, 
\omega_1\left(t^{-\alpha}\mathcal R^2\right)+g_1\left(t^{-\alpha}\mathcal R^2, A+t^{-\frac{\alpha}{2}}\mathcal E, \mathcal S+\frac{\kappa}{\varkappa}\varphi,\varphi,t\right) \\
& 
-\frac{\kappa}{\varkappa} \left(\omega_2\left(A+t^{-\frac{\alpha}{2}}\mathcal E\right)+  g_2\left(t^{-\alpha}\mathcal R^2, A+t^{-\frac{\alpha}{2}}\mathcal E, \mathcal S+\frac{\kappa}{\varkappa}\varphi,\varphi,t\right)\right),\\
\mathcal G_2(\mathcal R,\mathcal E,\mathcal S,\varphi,t)\equiv 
& \, 
\omega_2\left(A+t^{-\frac{\alpha}{2}}\mathcal E\right)-\omega_2(A)+g_2\left(t^{-\alpha}\mathcal R^2, A+t^{-\frac{\alpha}{2}}\mathcal E, \mathcal S+\frac{\kappa}{\varkappa}\varphi,\varphi,t\right).
\end{align*}
From \eqref{fgas}, \eqref{rc} and \eqref{fgasik} it follows that 
\begin{align*}
\mathcal F_l(\mathcal R,\mathcal E,\mathcal S,\varphi,t) = \sum_{K=2}^\infty t^{-\frac{K\alpha}{2}} \mathcal F_{l,K}(\mathcal R,\mathcal E,\mathcal S,\varphi), \quad 
\mathcal G_l(\mathcal R,\mathcal E,\mathcal S,\varphi,t) = \sum_{K=1}^\infty t^{-\frac{K\alpha}{2}} \mathcal G_{l,K}(\mathcal R,\mathcal E,\mathcal S,\varphi)
\end{align*}
as $t\to\infty$, where the coefficients $\mathcal F_{l,K}(\mathcal R,\mathcal E,\mathcal S,\varphi)$ and $\mathcal G_{l,K}(\mathcal R,\mathcal E,\mathcal S,\varphi)$ are $2\pi$-periodic with respect to $\mathcal S$ and $2\pi\varkappa$-periodic with respect to $\varphi$, and can be easily calculated as follows 
\begin{align*}
\mathcal F_{1,1}(\mathcal R,\mathcal E,\mathcal S,\varphi) \equiv & \, 0, \\
\mathcal F_{1,K}(\mathcal R,\mathcal E,\mathcal S,\varphi)  \equiv &
 \sum_{(i,j,k)\in \mathcal X_K } \frac{\mathcal R^{1+i}\mathcal E^j}{2 j!}\partial_{E_2}^j f_{1,k,i+2}\left(A,\mathcal S+\frac{\kappa}{\varkappa}\varphi,\varphi\right)+ \delta_{K\alpha,2}\frac{\alpha \mathcal R}{2},\\
\mathcal F_{2,K}(\mathcal E_1,\mathcal E_2,\mathcal S,\varphi) \equiv  &
 \sum_{(i,j,k)\in \mathcal X_{K+1} } \frac{\mathcal R^{i}\mathcal E^j}{j!}\partial_{E_2}^j f_{2,k,i}\left(A,\mathcal S+\frac{\kappa}{\varkappa}\varphi,\varphi\right)+ \delta_{K\alpha,2}\frac{\alpha\mathcal E}{2},\\
\mathcal G_{1,K}(\mathcal R,\mathcal E,\mathcal S,\varphi)  \equiv &
 \sum_{(i,j,k)\in \mathcal X_{K} } \frac{\mathcal R^{i}\mathcal E^j}{j!}\partial_{E_2}^j \left(g_{1,k,i}\left(A,\mathcal S+\frac{\kappa}{\varkappa}\varphi,\varphi\right)-\frac{\kappa}{\varkappa}g_{2,k,i}\left(A,\mathcal S+\frac{\kappa}{\varkappa}\varphi,\varphi\right)\right) \\
&  +\frac{\mathcal R^{K}}{(K/2)!}\partial^{K/2}_{E_1}\omega_1(0) - \frac{\kappa\mathcal E^K}{\varkappa K!}\partial^K_{E_2}\omega_2(A),\\
\mathcal G_{2,K}(\mathcal R,\mathcal E,\mathcal S,\varphi)  \equiv &
 \sum_{(i,j,k)\in \mathcal X_{K} } \frac{\mathcal R^{i}\mathcal E^j}{j!}\partial_{E_2}^j g_{2,k,i}\left(A,\mathcal S+\frac{\kappa}{\varkappa}\varphi,\varphi\right) + \frac{\mathcal E^K}{K!}\partial^K_{E_2}\omega_2(A).
\end{align*}
Here, $\mathcal X_K:=\{(i,j,k)\in\mathbb Z^3: i\geq 0, j\geq 0, k\geq 1, i+j+2k=K\}$ and it is assumed that $\partial^{K/2}_{E_1}\omega_1(0)\equiv 0$ for odd $K$. Note that system \eqref{EESp} is asymptotically autonomous with the corresponding limiting system
\begin{gather*}
\frac{d\hat {\mathcal R}}{dt}=0, \quad 
\frac{d\hat {\mathcal E}}{dt}=0, \quad 
\frac{d\hat {\mathcal S}}{dt}=0, \quad 
\frac{d\hat \varphi}{dt}=\omega_2(A).
\end{gather*}
In this case, $\varphi(t)$ can be considered as a fast variable as $t\to\infty$, and system \eqref{EESp} can be simplified by averaging the equations with respect to $\varphi$. Note that this method is widely used in the study of systems with a small parameter (see, for example,~\cite{BM61,AN84,BDP01,GL02,DM10}). The averaging transformation is sought in the following form
\begin{gather}\label{UVPsi}
\begin{split}
U(\mathcal R,\mathcal E,\mathcal S,\varphi,t)=\mathcal R+\sum_{K=1}^N t^{-\frac{K\alpha}{2}} U_{K}(\mathcal R,\mathcal E,\mathcal S,\varphi), \\
V(\mathcal R,\mathcal E,\mathcal S,\varphi,t)=\mathcal E+\sum_{K=1}^N t^{-\frac{K\alpha}{2}} V_{K}(\mathcal R,\mathcal E,\mathcal S,\varphi), \\
\Psi(\mathcal R,\mathcal E,\mathcal S,\varphi,t)=\mathcal S+\sum_{K=1}^N t^{-\frac{K\alpha}{2}} \Psi_{K}(\mathcal R,\mathcal E,\mathcal S,\varphi),
\end{split}
\end{gather}
where $U_K$, $V_K$, and $\Psi_K$ are assumed to be periodic with respect to $\mathcal S$ and $\varphi$. These coefficients are chosen in such a way that the system written in the new variables 
\begin{gather}\label{chrp}
\begin{split}
&\rho(t)=U(\mathcal R(t),\mathcal E(t),\mathcal S(t),\varphi(t),t), \\ 
&v(t)=V(\mathcal R(t),\mathcal E(t),\mathcal S(t),\varphi(t),t), \\ 
&\psi(t)=\Psi(\mathcal R(t),\mathcal E(t),\mathcal S(t),\varphi(t),t) 
\end{split}
\end{gather}
takes the form \eqref{rvpp}, where the right-hand side does not depend explicitly on $\varphi$, at least in the first $N$ terms of the asymptotics as $t\to\infty$. Note that both the transformation coefficients $U_K$, $V_K$, $\Psi_K$ and the coefficients $\Pi_K$, $\Lambda_K$, $\Omega_K$ of system \eqref{rvpp} must be determined. 
Consider the total derivative of the transformation \eqref{UVPsi} with respect to $t$:
\begin{gather*}
\frac{d}{dt}\begin{pmatrix} U  \\ V \\ \Psi \end{pmatrix}:= \left(\frac{d\mathcal R}{dt}\partial_{\mathcal R}+\frac{d\mathcal E}{dt}\partial_{\mathcal E}+\frac{d\mathcal S}{dt}\partial_{\mathcal S}+\frac{d\varphi}{dt}\partial_{\varphi}+\partial_t\right)\begin{pmatrix} U  \\ V \\ \Psi \end{pmatrix}.
\end{gather*}
It can easily be checked that
\begin{gather}\label{duvpsi}
\begin{split}
\frac{d}{dt}\begin{pmatrix} U  \\ V \\ \Psi \end{pmatrix}\sim &
\sum_{K=1}^\infty t^{-\frac{K\alpha}{2}} \left\{\begin{pmatrix} \mathcal F_{1,K} \\ \mathcal F_{2,K} \\ \mathcal G_{1,K} \end{pmatrix}+\omega_2(A)\partial_\varphi \begin{pmatrix} U_{K} \\ V_{K} \\ \Psi_K \end{pmatrix}+\left(1-\frac{K\alpha}{2}\right) \begin{pmatrix} U_{K-\frac{2}{\alpha}} \\ V_{K-\frac{2}{\alpha}} \\ \Psi_{K-\frac{2}{\alpha}} \end{pmatrix} \right\} \\ &
 +\sum_{K=2}^\infty t^{-\frac{K\alpha}{2}} \sum_{j=1}^{K-1}( \mathcal F_{1,j}\partial_{\mathcal R}+\mathcal F_{2,j}\partial_{\mathcal E} + \mathcal G_{1,j}\partial_{\mathcal S}+\mathcal G_{2,j} \partial_\varphi)\begin{pmatrix} U_{K-j} \\ V_{K-j} \\ \Psi_{K-j} \end{pmatrix}
\end{split}
\end{gather}
as $t\to\infty$, where it is assumed that $U_{k}\equiv V_{k}\equiv \Psi_k\equiv 0$ for $k\not\in\{1,\dots,N\}$. By grouping terms with the same powers of $t$ in \eqref{LLO} and \eqref{duvpsi}, we obtain the following system of differential equations 
\begin{gather}\label{ukvkpsik}
\omega_2(A)\partial_\varphi 
\begin{pmatrix} U_{K} \\ V_{K} \\ \Psi_K\end{pmatrix} =
\begin{pmatrix}
\Pi_{K}-\mathcal F_{1,K} +\tilde {\mathcal F}_{1,K}\\ 
\Lambda_{K}-\mathcal F_{2,K} +\tilde {\mathcal F}_{2,K}\\ 
\Omega_{K}-\mathcal G_{1,K} +\tilde {\mathcal G}_{1,K}\\ 
\end{pmatrix}, \quad K=1,\dots, N,
\end{gather} 
where $\tilde {\mathcal F}_{i,K}(\mathcal R,\mathcal E,\mathcal S,\varphi)$ and $\tilde {\mathcal G}_{1,K}(\mathcal R,\mathcal E,\mathcal S,\varphi)$ are expressed through $\{U_{j}, V_{j}, \Psi_j, \Pi_{j}, \Lambda_{j}, \Omega_{j}\}_{j=1}^{K-1}$. For instance,
\begin{align*}
	\begin{pmatrix} \tilde {\mathcal F}_{1,1} \\ \tilde {\mathcal F}_{2,1} \\ \tilde{\mathcal G}_{1,1} \end{pmatrix} 
	\equiv  & 
	\begin{pmatrix} 0\\ 0 \\ 0 \end{pmatrix},\\
	\begin{pmatrix} \tilde {\mathcal F}_{1,2} \\ \tilde {\mathcal F}_{2,2} \\ \tilde{\mathcal G}_{1,2} \end{pmatrix} 
	\equiv & 
	-(\mathcal F_{1,1}\partial_{\mathcal R}+\mathcal F_{2,1}\partial_{\mathcal E} + \mathcal G_{1,1}\partial_{\mathcal S}+\mathcal G_{2,1} \partial_\varphi)\begin{pmatrix}  U_{1} \\ V_{1} \\ \Psi_1  \end{pmatrix} - \left(1-\alpha\right) \begin{pmatrix} U_{2-\frac{2}{\alpha}} \\ V_{2-\frac{2}{\alpha}} \\ \Psi_{2-\frac{2}{\alpha}} \end{pmatrix} \\
	& + (U_{1}\partial_{\mathcal R}+V_{1}\partial_{\mathcal E}+\Psi_{1}\partial_{\mathcal S} ) \begin{pmatrix} \Pi_{1} \\ \Lambda_{1} \\ \Omega_1 \end{pmatrix} \\
	\begin{pmatrix} \tilde {\mathcal F}_{1,3} \\ \tilde {\mathcal F}_{2,3} \\ \tilde{\mathcal G}_{1,3} \end{pmatrix} 
	\equiv & 
	-\sum_{j=1}^2 (\mathcal F_{1,j}\partial_{\mathcal R}+\mathcal F_{2,j}\partial_{\mathcal E} + \mathcal G_{1,j}\partial_{\mathcal S}+\mathcal G_{2,j} \partial_\varphi)\begin{pmatrix} U_{3-j} \\ V_{3-j} \\ \Psi_{3-j}\end{pmatrix} -\left(1- \frac{3\alpha}{2}\right) \begin{pmatrix} U_{3-\frac{2}{\alpha}} \\ V_{3-\frac{2}{\alpha}} \\ \Psi_{3-\frac{2}{\alpha}} \end{pmatrix}\\
	& +\sum_{j=1}^2(U_{j}\partial_{\mathcal R}+V_{j}\partial_{\mathcal E}+ \Psi_{j}\partial_{\mathcal S}) \begin{pmatrix}  \Pi_{3-j}\\ \Lambda_{3-j} \\ \Omega_{3-j} \end{pmatrix} \\
	& + \frac{1}{2}\left(U_{1}^2\partial_{\mathcal R}^2+V_{1}^2 \partial_{\mathcal E}^2+\Psi_{1}^2 \partial_{\mathcal S}^2
	+2 \left(U_{1}V_{1} \partial_{\mathcal R}\partial_{\mathcal E}+U_{1}\Psi_{1} \partial_{\mathcal R}\partial_{\mathcal S}+ V_{1}\Psi_{1} \partial_{\mathcal E}\partial_{\mathcal S}\right)
	\right)\begin{pmatrix} \Pi_{1}\\ \Lambda_{1} \\ \Omega_1 \end{pmatrix}.
\end{align*}
To ensure periodicity of solutions of system \eqref{ukvkpsik}, we take
\begin{align*}
\Pi_{K}(\mathcal R,\mathcal E,\mathcal S)&\equiv \langle \mathcal F_{1,K}(\mathcal R,\mathcal E,\mathcal S,\varphi)-\tilde {\mathcal F}_{1,K}(\mathcal R,\mathcal E,\mathcal S,\varphi)\rangle_{ \varkappa \varphi},\\
\Lambda_{K}(\mathcal R,\mathcal E,\mathcal S)&\equiv \langle \mathcal F_{2,K}(\mathcal R,\mathcal E,\mathcal S,\varphi)-\tilde {\mathcal F}_{2,K}(\mathcal R,\mathcal E,\mathcal S,\varphi)\rangle_{ \varkappa \varphi},\\
\Omega_K(\mathcal R,\mathcal E,\mathcal S)&\equiv \langle \mathcal G_{1,K}(\mathcal R,\mathcal E,\mathcal S,\varphi)-\tilde {\mathcal G}_{1,K}(\mathcal R,\mathcal E,\mathcal S,\varphi)\rangle_{ \varkappa \varphi}.
\end{align*}
In this case, the right-hand side of \eqref{ukvkpsik} has zero mean. Hence, system \eqref{ukvkpsik} is solvable in the class of functions that are $2\pi\varkappa$-periodic in $\varphi$ with zero mean. Moreover, $U_{K}(\mathcal R,\mathcal E,\mathcal S,\varphi)$, $V_{K}(\mathcal R,\mathcal E,\mathcal S,\varphi)$, $\Psi_K(\mathcal R,\mathcal E,\mathcal S,\varphi)$, $\Pi_K(\mathcal R,\mathcal E,\mathcal S)$, $\Lambda_{K}(\mathcal R,\mathcal E,\mathcal S)$, and $\Omega_k(\mathcal R,\mathcal E,\mathcal S)$ are $2\pi$-periodic in $\mathcal S$. 
It can easily be checked that 
$\Pi_{1}(\mathcal R,\mathcal E,\mathcal S)\equiv 0$, 
$\Pi_{K}(\mathcal R,\mathcal E,\mathcal S)=\mathcal O(\mathcal R)$, 
$U_{K}(\mathcal R,\mathcal E,\mathcal S,\varphi)=\mathcal O(\mathcal R)$ as $\mathcal R\to 0$.
Thus, for all $0<\epsilon<\epsilon_\ast$ there exists $t_\ast\geq t_0$ such that
\begin{align*}
&|U(\mathcal R,\mathcal E,\mathcal S,\varphi,t)-\mathcal R|< \epsilon \mathcal R, &
&|V(\mathcal R,\mathcal E,\mathcal S,\varphi,t)-\mathcal E|< \epsilon, \\
&|\Psi(\mathcal R,\mathcal E,\mathcal S,\varphi,t)-\mathcal S|< \epsilon, & 
&|D(\mathcal R,\mathcal E,\mathcal S,\varphi,t)-1|< \epsilon
\end{align*}
for all $(\mathcal R,\mathcal E,\mathcal S,\varphi)\in\mathfrak D_{0,t_\ast}$ and $t\geq t_\ast$, where
\begin{gather*}
D(\mathcal E_1,\mathcal E_2,\mathcal S,\varphi,t):= 
\begin{vmatrix} 
\partial_{\mathcal R} U & \partial_{\mathcal E} U  & \partial_{\mathcal S} U\\ 
\partial_{\mathcal R} V & \partial_{\mathcal E} V & \partial_{\mathcal S} V\\ 
\partial_{\mathcal R} \Psi & \partial_{\mathcal E} \Psi  & \partial_{\mathcal S} \Psi
\end{vmatrix}.
\end{gather*}
Hence, \eqref{chrp} is invertible and there exists the reverse transformation $\mathcal R= \rho+P(\rho,v,\psi,\varphi,t)$, $\mathcal E= v+Q(\rho,v,\psi,\varphi,t)$, $\mathcal S=\psi+Z(\rho,v,\psi,\varphi,t)$  such that $|P|\leq \epsilon \rho$, $|Q|\leq \epsilon$, $|Z|\leq \epsilon$ for all $(\rho,v,\psi,\varphi)\in\mathfrak D_{\epsilon,t_\ast}$ and $t\geq t_\ast$.
Define
\begin{align*}
\begin{pmatrix}
\tilde \Pi(\rho,v,\psi,\varphi,t)\\
\tilde \Lambda(\rho,v,\psi,\varphi,t)\\
\tilde \Omega (\rho,v,\psi,\varphi,t)\\
\end{pmatrix}
\equiv &
\left.\frac{d}{dt}\begin{pmatrix} U \\ V \\ \Psi \end{pmatrix}\right|_{ \substack{
\mathcal R= \rho+P, 
\mathcal E= v+Q,
\mathcal S= \psi+Z
}
} - \sum_{k=1}^N t^{-\frac{K\alpha}{2}} \begin{pmatrix}
\Pi_{K}(\rho,v,\psi)\\
\Lambda_{K}(\rho,v,\psi)\\
\Omega_K(\rho,v,\psi)
\end{pmatrix}
\end{align*}
and $\tilde{\mathcal G}(\rho,v,\psi,\varphi,t)\equiv \mathcal G_2(\rho+P,v+Q,\psi+Z,\varphi,t)$. Then, we obtain the estimates \eqref{tildeLLOP}.
\end{proof}

\section{Analysis of the truncated system}
\label{sec4}
\begin{proof}[Proof of Lemma~\ref{Lem0}]
Substituting 
$\hat\rho=z_1$, $\hat v=z_2$, $\hat\psi=\psi_0+z_3$ into \eqref{rvpls}, we obtain
\begin{gather}\label{zsys}
\frac{d{\bf z}}{dt}={\bf F}({\bf z},t), \quad {\bf F}({\bf z},t)\equiv \begin{pmatrix} t^{-\alpha}z_1 \pi(\psi_0+z_3) \\ t^{-\frac\alpha 2}\lambda(\psi_0+z_3) \\ t^{-\frac{\alpha}{2}} \nu z_2\end{pmatrix} \equiv {\bf J}_{\bf F}(t){\bf z}+\tilde {\bf F}({\bf z},t),
\end{gather}
where ${\bf z}=(z_1,z_2,z_3)^T$, ${\bf J}_{\bf F}(t)\equiv \nabla_{\bf z}( {\bf F}(0,t))^T$ is the Jacobian matrix of ${\bf F}({\bf z},t)$ at the point ${\bf z}=0$, and $\tilde {\bf F}({\bf z},t)\equiv {\bf F}({\bf z},t)-{\bf J}_{\bf F}(t){\bf z} = \mathcal O(|{\bf z}|^2)$ as $|{\bf z}|\to 0$. The roots of the equation $|{\bf J}_{\bf F}(t)- \sigma {\bf I}|=0$ are given by
\begin{gather*}
\sigma_0(t)=t^{-\alpha }\pi(\psi_0), \quad \sigma_\pm(t)=\pm t^{-\frac{\alpha}{2}}\sqrt{\lambda'(\psi_0)\nu}.
\end{gather*}
Hence, if $\lambda'(\psi_0)\nu>0$, then the eigenvalues $\sigma_+(t)$ and $\sigma_-(t)$ are real of different signs and the solution ${\bf z}(t)\equiv 0$ of system \eqref{zsys} is unstable. Similarly, if $\pi(\psi_0)>0$, then $\sigma_0(t)>0$ and there exist $c>0$ and $\delta>0$ such that $\pi(\psi_0+z_3)\geq c>0$ for all $|z_3|\leq \delta$. In this case, $d|z_1|/dt\geq c t^{-\alpha} |z_1|$ for all $(z_1,z_2)\in\mathbb R^2$, $|z_3|\leq \delta$ and $t\geq t_\ast$. Integrating the last inequality with respect to $t$, we see that $|z_1(t)|\to\infty$ as $t\to\infty$. Hence, there exists $C>0$ such that for all $\epsilon\in (0,\delta)$ the solution ${\bf z}(t)$ of system \eqref{zsys} with initial $|{\bf z}(t_s)|\leq \epsilon$ and $t_s\geq t_\ast$ satisfies $|{\bf z}(t_e)|\geq C$ at some $t_e>t_s$. It follows that the solution ${\bf z}(t)\equiv 0$  is unstable.

Now, let $\pi(\psi_0)<0$ and $\lambda'(\psi_0)\nu<0$. Then, $\sigma_0(t)<0$, while $\sigma_+(t)$ and $\sigma_-(t)$ are pure imaginary. Consider 
\begin{gather*}
L({\bf z})=\frac{z_1^2}{2}+\frac{|\nu|z_2^2}{2}-({\hbox{\rm sgn}}\,\nu)\int\limits_0^\psi \lambda(\psi_0+\varsigma)\,d\varsigma
\end{gather*} 
as a Lyapunov function candidate for system \eqref{zsys}. It can easily be checked that 
$L_-|{\bf z}|^2\leq L({\bf z})\leq L_+|{\bf z}|^2$ for all ${\bf z}\in\mathbb R^3$ with $L_-=\min\{1,|\nu|,|\lambda'(\psi_0)|\}/4$ and $L_+=\max\{1,|\nu|,|\lambda'(\psi_0)|\}$. Note that there exist $c>0$ and $\delta>0$ such that $\pi(\psi+z_3)\leq - c$ for all $|z_3|\leq \delta$. Hence, the derivative of $L({\bf z})$ along the trajectories of the system satisfies
$dL/dt=t^{-\alpha}z_1^2\pi(\psi_0+z_3)\leq -c t^{-\alpha} z_1^2\leq 0$
for all $(z_1,z_2)\in\mathbb R^2$, $|z_3|\leq \delta$ and $t\geq t_\ast$.
Thus, the solution ${\bf z}(t)\equiv 0$ of system \eqref{zsys} is stable.
\end{proof}

\section{Analysis of the model system}
\label{sec5}
\begin{proof}[Proof of Lemma~\ref{Lem1}]
The proof is divided into two parts.

{\bf 1}. Let $\vartheta(\psi_0)<0$. Then, the following three cases arise: irrational $\alpha\in (0,1)$, rational $\alpha\in (0,1)$, and $\alpha=1$.

{\bf 1.1}. If $\alpha\in (0,1)$ is irrational, consider the functions
\begin{gather}\label{anz}\begin{split}
\hat\rho_\ast(t)& =\sum_{n=1}^{M_1}\sum_{m=0}^{M_2} t^{-\frac{n\alpha}{2}-m(1-\alpha)} \rho_{n,m}, \quad
\hat v_\ast(t)  =\sum_{n=1}^{M_1}\sum_{m=0}^{M_2} t^{-\frac{n\alpha}{2}-m(1-\alpha)} v_{n,m}, \\
\hat\psi_\ast(t) & =\psi_0+\sum_{n=1}^{M_1}\sum_{m=0}^{M_2} t^{-\frac{n\alpha}{2}-m(1-\alpha)} \psi_{n,m}
\end{split}
\end{gather}
with $M_1,M_2\in\mathbb Z_+$ and some constant coefficients $\rho_{n,m}$, $v_{n,m}$, and $\psi_{n,m}$. 
Substituting \eqref{anz} into \eqref{rvpmod} and grouping terms with the same powers of $t$, we obtain the following system
\begin{gather}\label{sysk}
\begin{split}
-\pi(\psi_0)\rho_{n,m} & = \mathcal X_{n,m}(\rho_{1,0},v_{1,0},\psi_{1,0},\dots,\rho_{n-1,m-1},v_{n-1,m-1},\psi_{n-1,m-1}), \\
-\nu v_{n,m} & = \mathcal Y_{n,m}(\rho_{1,0},v_{1,0},\psi_{1,0},\dots,\rho_{n-1,m-1},v_{n-1,m-1},\psi_{n-1,m-1}), \\
-\lambda'(\psi_0)\psi_{n,m} & = \mathcal Z_{n,m}(\rho_{1,0},v_{1,0},\psi_{1,0},\dots,\rho_{n-1,m-1},v_{n-1,m-1},\psi_{n-1,m-1})
\end{split}
\end{gather}
for $n=1,\dots, M_1$ and $m=0,\dots,M_2$, where it is assumed that $\rho_{0,m}=v_{0,m}=\psi_{0,m}=0$. In particular, $\mathcal X_{1,0}=\Pi_3(0,0,\psi_0)$, $\mathcal Y_{1,0}=\nu_{2,0}(\psi_0)$, $\mathcal Z_{1,0}=0$,  $\mathcal X_{1,1}=\alpha \rho_{1,0}/2$, $\mathcal Y_{1,1}=0$, $\mathcal Z_{1,1}=0$, $\mathcal X_{2,0}=\Pi_4(0,0,\psi_0)+(\rho_{1,0}\partial_\rho +v_{1,0}\partial_v+\psi_{1,0}\partial_\psi)\Pi_3(0,0,v_0)+\pi'(\psi_0)\rho_{1,0}\psi_{1,0}$, $\mathcal Y_{2,0}=\Omega_3(0,0,\psi_0)+\nu_{2,0}'(\psi_0)\psi_{1,0}$, $\mathcal Z_{2,0}=\Lambda_3(0,0,\psi_0)+\lambda_{2,1}'(\psi_0)\rho_{1,0}+\lambda_{2,2}'(\psi_0)v_{1,0}+\lambda''(\psi_0){\psi_{1,0}^2}/{2}$.
Since $\pi(\psi_0)\nu\lambda'(\psi_0)\neq 0$, we see that system \eqref{sysk} is solvable and the functions $\hat\rho_\ast(t)$, $\hat v_\ast(t)$, and $\hat \psi_\ast(t)$ are defined. 
It can easily be checked that 
\begin{gather*}
\begin{split}
\mathcal N_\rho(t)&:=\frac{d\hat \rho_\ast}{dt}-\Pi(\hat \rho_\ast(t),\hat v_\ast(t),\hat \psi_\ast(t),t)=t^{-\beta}\left(\mathcal O(t^{-\frac{3\alpha}{2}})+\mathcal O(t^{-1})\right), \\
\mathcal N_v(t)&:=\frac{d\hat v_\ast}{dt}-\Lambda(\hat \rho_\ast(t),\hat v_\ast(t),\hat \psi_\ast(t),t)=t^{-\beta+\frac{\alpha}{2}}\left(\mathcal O(t^{-\frac{3\alpha}{2}})+\mathcal O(t^{-1})\right), \\
\mathcal N_\psi(t)&:=\frac{d\hat \psi_\ast}{dt}-\Omega(\hat \rho_\ast(t),\hat v_\ast(t),\hat \psi_\ast(t),t)=t^{-\beta+\frac{\alpha}{2}}\left(\mathcal O(t^{-\frac{3\alpha}{2}})+\mathcal O(t^{-1})\right)
\end{split}
\end{gather*}
as $t\to\infty$, where $\beta=M_1\alpha/2+M_2(1-\alpha)>0$. To prove the existence of solutions to system \eqref{rvpmod} with such asymptotic behaviour, consider the change of variables
\begin{gather}\label{subsz}
\hat \rho=\hat\rho_\ast(t)+t^{-\gamma}z_1, \quad 
\hat v= \hat v_\ast(t)+t^{-\gamma}z_2, \quad 
\hat \psi= \hat \psi_\ast(t)+t^{-\gamma}z_3
\end{gather}
with $\gamma=\beta-\alpha/2>0$. In the new variables ${\bf z}=(z_1,z_2,z_3)^T$, system \eqref{rvpmod} takes the form
\begin{gather}\label{zsys1}
\frac{d{\bf z}}{dt}={\bf F}({\bf z},t)+{\bf N}(t), 
\end{gather}
where 
\begin{gather*}
{\bf F}({\bf z},t)\equiv t^{\gamma}\left({\bf a}(\hat\rho_\ast(t)+t^{-\gamma}z_1,\hat v_\ast(t)+t^{-\gamma}z_2,\hat \psi_\ast(t)+t^{-\gamma}z_3,t)-{\bf a}(\hat\rho_\ast(t),\hat v_\ast(t),\hat \psi_\ast(t),t)\right)+\gamma t^{-1}{\bf z},\\
{\bf a}({\bf z},t) \equiv \begin{pmatrix}\Pi(z_1,z_2,z_3,t) \\ \Lambda(z_1,z_2,z_3,t) \\ \Omega(z_1,z_2,z_3,t)\end{pmatrix}, \quad 
{\bf N}(t)\equiv - t^{\gamma}\begin{pmatrix} 
 \mathcal N_\rho(t) \\ \mathcal N_{v}(t) \\ \mathcal N_\psi(t)
\end{pmatrix}.
\end{gather*}
By choosing $M_1$ and $M_2\in\mathbb Z_+$ sufficiently large, we see that
\begin{align*}
{\bf F}({\bf z},t)&=
t^{-\frac{\alpha}{2}} 
\begin{pmatrix}
0 \\
\lambda'(\hat\psi_\ast(t))z_3 \\
\nu z_2
\end{pmatrix} 
+
t^{-\alpha}  
\begin{pmatrix}
\pi(\psi_0)z_1 \\
\lambda_{2,1}(\psi_0)z_1 +\lambda_{2,2}(\psi_0)z_2 \\
\nu_{2,0}'(\psi_0)z_3 
\end{pmatrix}+\mathcal O(t^{-\frac{3\alpha}{2}})\mathcal O({\bf z})+\mathcal O(t^{-1})\mathcal O({\bf z}), \\ 
{\bf N}(t)&=\mathcal O(t^{-\frac{3\alpha}{2}})+\mathcal O(t^{-1})
\end{align*}
as  $|{\bf z}|\to 0$ and $t\to\infty$. Consider 
\begin{gather}\label{Lalpha}
\begin{split}
L_\alpha({\bf z},t)\equiv & \, \frac{z_1^2}{2}+\frac{|\nu|z_2^2}{2}-({\hbox{\rm sgn}}\,\nu) \lambda'(\hat\psi_\ast(t))\frac{z_3^2}{2}\\
&-({\hbox{\rm sgn}}\,\nu) t^{-\frac{\alpha}{2}}\left(\lambda_{2,1}(\psi_0)z_1z_3+  (\lambda_{2,2}(\psi_0)-\nu_{2,0}'(\psi_0))\frac{z_2z_3}{2} \right)
\end{split}
\end{gather}
as a Lyapunov function candidate for system \eqref{zsys1}. Note that there exist $d_1>0$ and $t_1\geq t_\ast$ such that 
$
L_- |{\bf z}|^2\leq L_\alpha({\bf z},t)\leq L_+ |{\bf z}|^2
$
for all $|{\bf z}|\leq d_1$ and $t\geq t_1$, where 
\begin{gather}\label{LLpm}
L_-=\frac{1}{4}\min\{1,|\nu|,|\lambda'(\psi_0)|\}, \quad 
L_+=\max\{1,|\nu|,|\lambda'(\psi_0)|\}.
\end{gather}
The derivative of $L_\alpha({\bf z},t)$ along the trajectories of system \eqref{zsys1} satisfies 
\begin{gather*}
\frac{dL_\alpha}{dt}=D_\alpha^1({\bf z},t)+D_\alpha^2({\bf z},t),
\end{gather*} 
where $D_\alpha^1({\bf z},t)\equiv \partial_t L_\alpha({\bf y},t)+({\bf F}({\bf z},t))^T \nabla_{\bf z} L_\alpha({\bf z},t)$ and 
$D_\alpha^2({\bf z},t)\equiv ({\bf N}(t))^T\nabla_{\bf z} L_\alpha({\bf z},t)$. 
It can easily be checked that
\begin{align*}
D_\alpha^1({\bf z},t)& = t^{-\alpha}\left(\pi(\psi_0)z_1^2+ (|\nu|z_2^2+|\lambda'(\psi_0)|z_3^2)\frac{\vartheta(\psi_0)}{2}\right)+\mathcal O(t^{-\frac{3\alpha}{2}})\mathcal O(|{\bf z}|^2)+\mathcal O(t^{-1})\mathcal O(|{\bf z}|^2), \\
D_\alpha^2({\bf z},t) &= \mathcal O(t^{-\frac{3\alpha}{2}})\mathcal O(|{\bf z}|)+\mathcal O(t^{-1})\mathcal O(|{\bf z}|)
\end{align*}
as $|{\bf z}|\to 0$ and $t\to\infty$. Hence, there exist $d_2>0$, $t_2\geq t_\ast$, $C_1>0$ such that 
\begin{gather*}
D_\alpha^1({\bf z},t)\leq - t^{-\alpha}\mu_\alpha |{\bf z}|^2, \quad 
D_\alpha^2({\bf z},t)\leq  C_1 t^{-\alpha-\sigma} |{\bf z}|
\end{gather*}
for all $|{\bf z}|\leq d_2$ and $t\geq t_1$, where 
\begin{gather}\label{mualpha}
\mu_\alpha=\frac{1}{4}\min\{2|\pi(\psi_0)|, |\nu \vartheta(\psi_0)|, |\lambda'(\psi_0)\vartheta(\psi_0)|\}, \quad 
\sigma=\min\left\{\frac{\alpha}{2},1-\alpha\right\}>0.
\end{gather}
Therefore, for all $\epsilon\in(0,\min\{d_1,d_2\})$ there exist 
\begin{gather*}
\delta_\epsilon=\frac{2C_1 t_\epsilon^{-\sigma}}{\mu_\alpha}, \quad 
t_\epsilon=\max\left\{t_1,t_2,\left(\frac{4C_1}{\mu_\alpha\epsilon}\sqrt{\frac{L_+}{L_-}}\right)^{\frac{1}{\sigma}}\right\}
\end{gather*}
such that 
\begin{gather}\label{dLineq}
\frac{dL_\alpha}{dt}\leq t^{-\alpha}\left(-\mu_\alpha + t_\epsilon^{-\sigma}\frac{C_1}{\delta_\epsilon}\right)|{\bf z}|^2\leq -t^{-\alpha}\frac{\mu_\alpha |{\bf z}|^2}{2}
\end{gather}
for all $\delta_\epsilon\leq |{\bf z}|\leq \epsilon$ and $t\geq t_\epsilon$. Moreover, the following inequalities hold
\begin{gather}\label{supinf}
\sup_{|{\bf z}|\leq \delta_\epsilon}L_\alpha({\bf z},t)\leq L_+ \delta^2_\epsilon<L_-\epsilon^2=\inf_{|{\bf z}|=\epsilon}L_\alpha({\bf z},t)
\end{gather}
for all $t\geq t_\epsilon$. Combining this with \eqref{dLineq}, we see that the solution ${\bf z}(t)$ of system \eqref{zsys1} with initial conditions $|{\bf z}(t_s)|\leq \delta_\epsilon$ and $t_s\geq t_\epsilon$ cannot leave the domain $\{|{\bf z}|\leq \epsilon\}$ for all $t\geq t_s$. From \eqref{subsz} it follows that there exists a solution to system \eqref{rvpmod} such that 
$\hat \rho_0(t)=\mathcal O(t^{-\alpha/2})$, $\hat v_0(t)=\mathcal O(t^{-\alpha/2})$, and $\hat \psi_0(t)=\psi_0+\mathcal O(t^{-\alpha})$ as $t\to\infty$.

{\bf 1.2}. If $\alpha=p/q$, where $p$ and $q$ are coprime integers such that $p<q$, consider the following asymptotic ansatz 
\begin{gather}\label{anz1}
\hat\rho_\ast(t) =\sum_{k=p}^{M} t^{-\frac{k}{2q}} \rho_{k}, \quad
\hat v_\ast(t)  = \sum_{k=p}^{M} t^{-\frac{k}{2q}} v_{k}, \quad
\hat\psi_\ast(t)  =\psi_0+\sum_{k=p}^{M} t^{-\frac{k}{2q}} \psi_{k}
\end{gather}
with $M\in\mathbb Z_+$, $M\geq p$ and constant coefficients $\rho_k$, $v_k$, and $\psi_k$. Substituting \eqref{anz1} into \eqref{rvpmod} and equating terms of like powers of $t$, we obtain $\rho_p=-\Pi_3(0,0,\psi_0)/\pi(\psi_0)$, $v_p=-\nu_{2,0}(\psi)/\nu$, $\psi_p=0$, and the following system
\begin{gather*}
-\pi(\psi_0)\rho_{p+k}  = \mathcal X_{p+k}, \quad
-\nu v_{p+k}  = \mathcal Y_{p+k}, \quad
-\lambda'(\psi_0)\psi_{p+k}  = \mathcal Z_{p+k}
\end{gather*}
for $k\geq 1$, where the right-hand sides are expressed through $\rho_p,v_p,\psi_p,\dots,\rho_{p+k-1},v_{p+k-1}$, and $\psi_{p+k-1}$. For example, if $p<q<3p/2$, then $\mathcal X_{p+k} = 0$ for $k< 2(q-p)$, $\mathcal X_{2q-p}= p\rho_p/(2q)$, $\mathcal X_{p+k} =  0$ for $2(q-p)< k<p$, $\mathcal X_{2p} =  \Pi_4(0,0,\psi_0)+(\rho_p\partial_\rho +v_p\partial_v+\psi_p\partial_\psi)\Pi_3(0,0,v_0)+\pi'(\psi_0)\rho_p\psi_p$,
$\mathcal Y_{p+k} =  0$ for $k<2p$, $\mathcal Y_{2p}=\Omega_3(0,0,\psi_0)$,
 $\mathcal Z_{p+k} = 0$ for  $k< 2q-p$, $\mathcal Z_{2q}= pv_p/(2q)$, $\mathcal Z_{p+k} =  0$ for $2q-p< k<p$, $\mathcal Z_{2p}=\Lambda_3(0,0,\psi_0)+\lambda_{2,1}'(\psi_0)\rho_p+\lambda_{2,2}'(\psi_0)v_p$. In this case, 
\begin{gather*}
\mathcal N_\rho(t)=\mathcal O\left(t^{-\frac{M+2p+1}{2q}}\right), \quad 
\mathcal N_v(t)=\mathcal O\left(t^{-\frac{M+p+1}{2q}}\right), \quad
\mathcal N_\psi(t)=\mathcal O\left(t^{-\frac{M+p+1}{2q}}\right), \quad t\to\infty.
\end{gather*}
Substituting \eqref{subsz} with $\gamma=(M-p)/(2q)$ and $M>p$ into \eqref{rvpmod}, we get system \eqref{zsys}, where
\begin{align*}
{\bf F}({\bf z},t)&=
t^{-\frac{p}{2q}} 
\begin{pmatrix} 0 \\
\lambda'(\hat\psi_\ast(t))z_3 \\
\nu z_2 
\end{pmatrix}
+
t^{-\frac{p}{q}} 
\begin{pmatrix} 
\pi(\psi_0) \\
\lambda_{2,1}(\psi_0)z_1 +\lambda_{2,2}(\psi_0)z_2
\\
\nu_{2,0}'(\psi_0)z_3 
\end{pmatrix}
+\mathcal O(t^{-\frac{2p+1}{2q}})\mathcal O({\bf z}), \\ 
{\bf N}(t)&=\mathcal O(t^{-\frac{2p+1}{2q}})
\end{align*}
as  $|{\bf z}|\to 0$ and $t\to\infty$. Consider the Lyapunov function \eqref{Lalpha} with $\alpha=p/q$. Then, the following estimates hold
\begin{align*}
D_{p/q}^1({\bf z},t)& = t^{-\frac{p}{q}}\left(\pi(\psi_0)z_1^2+ (|\nu| z_2^2+|\lambda'(\psi_0)| z_3^2)\frac{\vartheta(\psi_0)}{2}\right)+\mathcal O(t^{-\frac{2p+1}{2q}})\mathcal O(|{\bf z}|^2), \\
D_{p/q}^2({\bf z},t) &= \mathcal O(t^{-\frac{2p+1}{2q}})\mathcal O(|{\bf z}|)
\end{align*}
as $|{\bf z}|\to 0$ and $t\to\infty$. Hence, there exist $d_4>0$ and $t_4\geq t_\ast$ such that 
\begin{gather*}
D_{p/q}^1({\bf z},t)\leq - t^{-p/q}\mu_{p/q} |{\bf z}|^2, \quad D_{p/q}^2({\bf z},t)\leq  C_2 t^{-\frac{2p+1}{2q}} |{\bf z}|
\end{gather*}
 for all $|{\bf z}|\leq d_4$ and $t\geq t_4$, where $\mu_{p/q}=\mu_\alpha>0$ and $C_2$ is a positive constant. 
It follows that for all $\epsilon\in(0,\min\{d_1,d_4\})$ there exist 
\begin{gather*}
\delta_\epsilon=\frac{2C_2 t_\epsilon^{-\frac{1}{2q}}}{\mu_{p/q}}, \quad t_\epsilon=\max\left\{t_1,t_4,\left(\frac{4C_2}{\mu_{p/q}\epsilon}\right)^{2q} \frac{L_+}{L_-} \right\}
\end{gather*}
such that $dL_{p/q}/dt\leq -t^{-p/q}\mu_{p/q}|{\bf z}|^2/2$ for all $\delta_\epsilon\leq |{\bf z}|\leq \epsilon$ and $t\geq t_\epsilon$. Arguing as above, we see that there is a solution to system \eqref{rvpmod} such that $\hat \rho_0(t)=\mathcal O(t^{-\frac{p}{2q}})$, $\hat v_0(t)=\mathcal O(t^{-\frac{p}{2q}})$, and $\hat \psi_0(t)=\psi_0+\mathcal O(t^{-\frac{p}{q}})$ as $t\to\infty$. 

{\bf 1.3}. If $\alpha=1$, the asymptotic ansatz is considered in the form
\begin{gather}\label{anz2}
\hat\rho_\ast(t)=\sum_{k=1}^{M+1} t^{-\frac{k}{2}} \rho_k(\log t), \quad
\hat v_\ast(t)=\sum_{k=1}^{M+1} t^{-\frac{k}{2}} v_k(\log t), \quad 
\hat\psi_\ast(t)=\psi_0+\sum_{k=1}^{M+1} t^{-\frac{k}{2}} \psi_k(\log t)
\end{gather}
with $M\in\mathbb Z_+$, where $\rho_k(\tau)$, $v_k(\tau)$, and $\psi_k(\tau)$ are some polynomials. Substituting \eqref{anz2} into \eqref{rvpmod} yields the following equations for the coefficients
\begin{gather}\label{sysk2}
\begin{split}
\frac{d\rho_k}{d\tau}-\left(\frac{k}{2}+\pi(\psi_0)\right)\rho_k & = \mathcal X_k, \\
-\nu v_k & = \mathcal Y_k-\frac{d\psi_{k-1}}{d\tau}+\frac{(k-1)\psi_{k-1}}{2}, \\
-\lambda'(\psi_0)\psi_k & = \mathcal Z_k-\frac{dv_{k-1}}{d\tau}+\frac{(k-1)v_{k-1}}{2}
\end{split}
\end{gather}
for $k\geq 1$, where $\mathcal X_1=\Pi_3(0,0,\psi_0)$, $\mathcal Y_1=\nu_{2,0}(\psi_0)$, $\mathcal Z_1=0$, and the right-hand sides for $k\geq 2$ are some polynomials in $\rho_1,v_1,\psi_1,\dots,\rho_{k-1},v_{k-1},\psi_{k-1}$. For instance,
\begin{align*}
\mathcal X_2=&\, \Pi_4(0,0,\psi_0)+(\rho_1\partial_\rho +v_1\partial_v+\psi_1\partial_\psi)\Pi_3(0,0,v_0)+\pi'(\psi_0)\rho_1\psi_1,\\  
\mathcal Y_2=&\,\Omega_3(0,0,\psi_0)+\nu_{2,0}'(\psi_0)\psi_1,\\ 
\mathcal Z_2=&\,\Lambda_3(0,0,\psi_0)+\lambda_{2,1}'(\psi_0)\rho_1+\lambda_{2,2}'(\psi_0)v_1+\lambda''(\psi_0)\frac{\psi_1^2}{2},\\ 
\mathcal X_3=&\, \Pi_5(0,0,\psi_0)+\sum_{l=1}^2(\rho_l\partial_\rho +v_l\partial_v+\psi_l\partial_\psi)\Pi_{5-l}(0,0,v_0)+\pi'(\psi_0)(\rho_2\psi_1+\rho_1\psi_2)\\
&\,+ \left(\frac{\rho_1^2}{2}\partial^2_\rho +\frac{v_1^2}{2}\partial_v^2+\frac{\psi_1^2}{2}\partial_\psi^2+\rho_1 v_1\partial_\rho\partial_v+\rho_1\psi_1\partial_\rho\partial_\psi+v_1\psi_1\partial_v\partial_\psi\right)\Pi_{3}(0,0,v_0)+\pi''(\psi_0)\frac{\rho_1\psi_1^2}{2},\\  
\mathcal Y_3=&\,\Omega_4(0,0,\psi_0)+(\rho_1\partial_\rho +v_1\partial_v+\psi_1\partial_\psi)\Omega_3(0,0,\psi_0)+\nu_{2,0}'(\psi_0)\psi_2+\nu_{2,0}''(\psi_0)\frac{\psi_1^2}{2}\\ 
&\,+\nu_{2,1}\rho_1^2+\nu_{2,2}v_1^2,\\  
 \mathcal Z_3=&\,\Lambda_4(0,0,\psi_0)+(\rho_1\partial_\rho +v_1\partial_v+\psi_1\partial_\psi)\Lambda_3(0,0,\psi_0)+\lambda_{2,1}'(\psi_0)\rho_2+\lambda_{2,2}'(\psi_0)v_2\\ 
& \,+(\lambda_{2,1}'(\psi_0)\rho_1+\lambda_{2,2}'(\psi_0)v_1)\psi_1+\lambda'''(\psi_0)\frac{\psi_1^3}{6}.
\end{align*} 
Define
\begin{gather*}
\rho_k(\tau)\equiv
\begin{cases}
\displaystyle -\frac{2\mathcal X_k}{k+2\pi(\psi_0)}& \text{if} \quad k<-2\pi (\psi_0),\\ 
\displaystyle \tau \mathcal X_k & \text{if} \quad k=-2\pi (\psi_0), \\
\displaystyle \int\limits_\infty^\tau e^{-\frac{(k+2\pi(\psi_0))(\varsigma-\tau)}{2}}\mathcal X_k\,d\varsigma & \text{if} \quad k>-2\pi (\psi_0).
\end{cases}
\end{gather*}
Then, system \eqref{sysk2} is solvable, and the asymptotic solution \eqref{anz2} can be constructed. In this case, for all $M\geq 1$ the following estimates hold
\begin{gather*}
\mathcal N_\rho(t)=\mathcal O\left(t^{-\frac{M+3}{2}}\right), \quad 
\mathcal N_v(t)=\mathcal O\left(t^{-\frac{M+2}{2}}\right), \quad 
\mathcal N_\psi(t)=\mathcal O\left(t^{-\frac{M+2}{2}}\right), \quad t\to\infty. 
\end{gather*} 
Substituting \eqref{subsz} with $\gamma=\chi$, $\chi\in (0,\chi_0)$ into \eqref{rvpmod}, we obtain system \eqref{zsys}, where 
\begin{align*}
{\bf F}({\bf z},t)&=t^{-\frac{1}{2}} 
\begin{pmatrix} 
0 \\
\lambda'(\hat\psi_\ast(t))z_3 \\
\nu z_2
\end{pmatrix}
+ t^{-1} 
\begin{pmatrix}  (\chi+\pi(\psi_0))z_1 \\
\lambda_{2,1}(\psi_0)z_1 +(\chi+\lambda_{2,2}(\psi_0))z_2 \\
(\chi+\nu_{2,0}'(\psi_0))z_3 
\end{pmatrix}+\mathcal O(t^{-\frac{3}{2}})\mathcal O({\bf z}), \\
{\bf N}(t)&=\mathcal O(t^{-\frac{M+1}{2}})
\end{align*}
as  $|{\bf z}|\to 0$ and $t\to\infty$. The Lyapunov function is considered in the form \eqref{Lalpha} with $\alpha=1$. In this case, 
\begin{align*}
D_1^1({\bf z},t)& = t^{-1}\left((\chi+\pi(\psi_0))z_1^2+ (|\nu|z_2^2+|\lambda'(\psi_0)|z_3^2)\left(\chi+\frac{\vartheta(\psi_0)}{2}\right)\right)+\mathcal O(t^{-\frac{3}{2}})\mathcal O(|{\bf z}|^2), \\
D_1^2({\bf z},t) &= \mathcal O(t^{-\frac{3}{2}})\mathcal O(|{\bf z}|)
\end{align*}
as $|{\bf z}|\to 0$ and $t\to\infty$. It follows that there exist $d_3>0$ and $t_3\geq t_\ast$ such that 
\begin{gather*}
D_1^1({\bf z},t)\leq - t^{-1}\mu_1 |{\bf z}|^2, \quad 
D_1^2({\bf z},t)\leq  C_3 t^{-\frac{3}{2}} |{\bf z}|
\end{gather*}
for all $|{\bf z}|\leq d_3$ and $t\geq t_3$, where $\mu_1=\min\{|\chi+\pi(\psi_0)|, |\nu||\chi+\vartheta(\psi_0)/2|, |\lambda'(\psi_0)||\chi+\vartheta(\psi_0)/2|\}/2$ and $C_3$ is a positive constant. 
Therefore, for all $\epsilon\in(0,\min\{d_1,d_3\})$ there exist 
\begin{gather*}
\delta_\epsilon=\frac{2C_3 t_\epsilon^{-\frac{1}{2}}}{\mu_1}, \quad t_\epsilon=\max\left\{t_1,t_3,\left(\frac{4C_3}{\mu_1\epsilon}\right)^{2} \frac{L_+}{L_-} \right\}
\end{gather*}
such that $dL_1/dt\leq -t^{-1}\mu_1 |{\bf z}|^2/2$ for all $\delta_\epsilon\leq |{\bf z}|\leq \epsilon$ and $t\geq t_\epsilon$. Repeating the above reasoning, we conclude that there is a solution to system \eqref{rvpmod} such that $\hat \rho_0(t)=\mathcal O(t^{-\chi})$, $\hat v_0(t)=\mathcal O(t^{-\chi})$, $\hat \psi_0(t)=\psi_0+\mathcal O(t^{-\chi})$ as $t\to\infty$. 

To prove the stability of the solution $(\hat\rho_0(t),\hat v_0(t),\hat \psi_0(t))$, consider the following change of variables in system \eqref{rvpmod}
\begin{gather*}
\hat\rho = \hat \rho_0(t)+y_1, \quad 
\hat v = \hat v_0(t)+y_2, \quad 
\hat\psi = \hat \psi_0(t)+y_3.
\end{gather*}
Then, in the new variables ${\bf y}=(y_1,y_2,y_3)^T$ the system takes the form
\begin{gather}\label{ysys0}
\frac{d{\bf y}}{dt}={\bf G}({\bf y},t), \quad
{\bf G}({\bf y},t)\equiv {\bf a}(\hat \rho_0(t)+y_1,\hat v_0(t)+y_2,\hat \psi_0(t)+y_3,t)-{\bf a}(\hat \rho_0(t),\hat v_0(t),\hat \psi_0(t),t).
\end{gather}
Note that ${\bf G}(0,t)\equiv 0$ and
\begin{gather}\label{Gest}
\begin{split}
{\bf G}({\bf y},t)=&\, t^{-\frac{\alpha}{2}} 
\begin{pmatrix} 
0 \\
\lambda(\hat\psi_0(t)+y_3)-\lambda(\hat\psi_0(t)) \\
\nu y_2
\end{pmatrix} \\
&
+ t^{-\alpha} 
\begin{pmatrix}  \pi(\psi_0+y_3)y_1 \\
\lambda_{2,1}(\psi_0+y_3)y_1 + \lambda_{2,2}(\psi_0+y_3) y_2 \\
 \nu_{2,0}(\psi_0+y_3) -\nu_{2,0}(\psi_0)+y_1^2\nu_{2,1}+y_2^2\nu_{2,2}
\end{pmatrix}+\mathcal O(t^{-\alpha-\varsigma})\mathcal O({\bf y})
\end{split}
\end{gather}
as  $|{\bf y}|\to 0$ and $t\to\infty$, where 
\begin{gather}\label{varsigma}
\varsigma=\begin{cases}
\frac{\alpha}{2}, & \text{if} \quad \alpha\neq 1,\\
\chi\in (0,\chi_0), & \text{if} \quad \alpha = 1.
\end{cases}
\end{gather}
Consider
\begin{gather}
\begin{split}\label{Lfunct}
L({\bf y},t)\equiv & \, \frac{y_1^2}{2}+\frac{|\nu|y_2^2}{2}-({\hbox{\rm sgn}}\,\nu) \int\limits_0^{y_3}\left(\lambda(\hat\psi_0(t)+\zeta)-\lambda(\hat\psi_0(t))\right)d\zeta \\
&-({\hbox{\rm sgn}}\,\nu) t^{-\frac{\alpha}{2}}\left(\lambda_{2,1}(\psi_0)y_1y_3+  (\lambda_{2,2}(\psi_0)-\nu_{2,0}'(\psi_0))\frac{y_2y_3}{2} \right)
\end{split}
\end{gather}
as a Lyapunov function candidate for system \eqref{ysys0}. It can easily be checked that $L({\bf y},t)=(y_1^2+|\nu|y_2^2+|\lambda'(\psi_0)|y_3^2)/2+\mathcal O(|{\bf y}|^3)+\mathcal O(t^{-\varsigma})\mathcal O(|{\bf y}|^2)$ as $|{\bf y}|\to 0$ and $t\to\infty$. Hence, there exist $\Delta_1>0$ and $\tau_1\geq t_\ast$ such that $L_- |{\bf y}|^2\leq L({\bf y},t)\leq L_+ |{\bf y}|^2$ for all $|{\bf y}|\leq \Delta_1$ and $t\geq \tau_1$. The derivative of $L({\bf y},t)$ along the trajectories of system \eqref{ysys0} satisfies 
\begin{align*}
\frac{dL}{dt}& =\partial_t L({\bf y},t)+({\bf G}({\bf y},t))^T \nabla_{\bf y} L({\bf y},t) \\
& = t^{-\alpha}\left(\pi(\psi_0)y_1^2+ (|\nu|y_2^2+|\lambda'(\psi_0)|y_3^2)\frac{\vartheta(\psi_0)}{2}+\mathcal O(|{\bf y}|^3)+\mathcal O(t^{-\varsigma})\mathcal O(|{\bf y}|^2)\right)
\end{align*} 
as $|{\bf y}|\to 0$ and $t\to\infty$. Hence, there exist $\Delta_2\in (0,\Delta_1)$ and $\tau_2\geq \tau_1$ such that 
\begin{gather*}
\frac{dL}{dt}\leq -t^{-\alpha}\mu_\alpha |{\bf y}|^2\leq - t^{-\alpha}c L
\end{gather*}
for all $|{\bf y}|\leq \Delta_2$ and $t\geq \tau_2$ with $c=\mu_\alpha/L_+>0$. Hence, the solution ${\bf y}(t)\equiv 0$ is stable. Moreover, integrating the last inequality as $\tau\geq \tau_2$ with $|{\bf y}(\tau_2)|\leq \Delta_2$, we get
\begin{align*}
|{\bf y}(t)|&\leq  B\exp\left(-\frac{c}{2(1-\alpha)}\left(t^{1-\alpha}-\tau_2^{1-\alpha}\right)\right), & \alpha&<1, \\
|{\bf y}(t)|&\leq B\left(\frac{t}{\tau_2}\right)^{-\frac{c}{2}}, & \alpha&=1,
\end{align*}
where $B=\Delta_2\sqrt{L_+/L_-}>0$. It follows that both the solution ${\bf y}(t)\equiv 0$ of system \eqref{ysys0} and the solution $(\hat\rho_0(t),\hat v_0(t),\hat\psi_0(t))$ of system \eqref{rvpmod} are asymptotically stable. 

{\bf 2}. Now, let $\vartheta(\psi_0)>0$. Substituting 
\begin{gather*}
\hat\rho = \hat \rho_\ast(t)+y_1, \quad 
\hat v = \hat v_\ast(t)+y_2, \quad 
\hat\psi = \hat \psi_\ast(t)+y_3
\end{gather*}
with some sufficiently large $M,M_1,M_2\in\mathbb Z_+$ into \eqref{rvpmod}, we obtain
\begin{gather}\label{ysys1}
\frac{d{\bf y}}{dt}={\bf G}_\ast({\bf y},t)+{\bf N}_\ast(t),
\end{gather}
where
\begin{align}
\label{Gast}
{\bf G}_\ast({\bf y},t) \equiv&\, {\bf a}(\hat \rho_\ast(t)+y_1,\hat v_\ast(t)+y_2,\hat \psi_\ast(t)+y_3,t)-{\bf a}(\hat \rho_\ast(t),\hat v_\ast(t),\hat \psi_\ast(t),t) \\ 
\nonumber
=&\, t^{-\frac{\alpha}{2}} 
\begin{pmatrix} 
0 \\
\lambda(\hat\psi_\ast(t)+y_3)-\lambda(\hat\psi_\ast(t)) \\
\nu y_2
\end{pmatrix} \\
\nonumber
&
+ t^{-\alpha} 
\begin{pmatrix}  \pi(\psi_0+y_3)y_1 \\
\lambda_{2,1}(\psi_0+y_3)y_1 + \lambda_{2,2}(\psi_0+y_3) y_2 \\
 \nu_{2,0}(\psi_0+y_3) -\nu_{2,0}(\psi_0)+y_1^2\nu_{2,1}+y_2^2\nu_{2,2}
\end{pmatrix}+\mathcal O(t^{-\frac{3\alpha}{2}})\mathcal O({\bf y}),\\
\label{Nast}
{\bf N}_\ast(t)  \equiv &\, - \begin{pmatrix} 
 \mathcal N_\rho(t) \\ \mathcal N_{v}(t) \\ \mathcal N_\psi(t)
\end{pmatrix} = \mathcal O(t^{-\frac{3\alpha}{2}})
\end{align}
as  $|{\bf y}|\to 0$ and $t\to\infty$.
Let us first show that $y_1(t)$ remains near zero as long as $\sqrt{y_2^2(t)+y_3^2(t)}\leq d$ with some $d>0$.
Consider $\ell(y_1)=y_1^2/2$ as a Lyapunov function. It easily follows that $d\ell/dt=t^{-\alpha} \pi(\psi_0+y_3)y_1^2+\mathcal O(y_1)\mathcal O(t^{-3\alpha/2})$ as $y_1\to 0$ and $t\to\infty$. Since $\pi(\psi_0)<0$, there exists $\tilde d>0$ such that $\pi(\psi_0+y_3)\leq -|\pi(\psi_0)|/2$ for all $|y_3|\leq \tilde d$. Hence, there exist $t_1\geq t_\ast$, $d_1>0$, and $B_1>0$ such that $d\ell/dt\leq t^{-\alpha}(-|\pi(\psi_0)|y_1^2/2+B_1 t^{-\alpha} |y_1|)$ for $|y_1|\leq d_1$, $|\tilde{\bf y}|:=\sqrt{y_2^2+y_3^2}\leq \tilde d$, and $t\geq t_1$. Moreover, for all $\epsilon\in (0,d_1]$, there exist
\begin{gather*}
\delta_\epsilon=\frac{4 B_1\tau_\epsilon^{-\frac{\alpha}{2}}}{|\pi(\psi_0)|}, \quad \tau_\epsilon=\max\left\{t_1, \left(\frac{8B_1}{|\pi(\psi_0) |\epsilon}\right)^{\frac{2}{\alpha}}\right\}
\end{gather*}
such that $d\ell/dt\leq t^{-\alpha}(-|\pi(\psi_0)|/2+B_1\tau_\epsilon^{-\alpha/2}\delta_\epsilon^{-1})y_1^2\leq 0$ for all $\delta_\epsilon\leq |y_1|\leq \epsilon $, $|\tilde{\bf y}|\leq \tilde d$ and $t\geq \tau_\epsilon$. 
Next, we show that instability occurs with respect to the variables $y_2$ and $y_3$.  Consider 
\begin{align*}
\tilde \ell(y_2,y_3,t)\equiv & \, \frac{|\nu|y_2^2}{2}-({\hbox{\rm sgn}}\,\nu) \int\limits_0^{y_3}\left(\lambda(\hat\psi_\ast(t)+\zeta)-\lambda(\hat\psi_\ast(t))\right)d\zeta  -({\hbox{\rm sgn}}\,\nu) t^{-\frac{\alpha}{2}}  (\lambda_{2,2}(\psi_0)-\nu_{2,0}'(\psi_0))\frac{y_2y_3}{2}  
\end{align*}
as a Lyapunov function candidate for system \eqref{ysys1}. Note that 
$\tilde \ell(y_2,y_3,t)=(|\nu|y_2^2+|\lambda'(\psi_0)|y_3^2)/2+\mathcal O(|\tilde {\bf y}|^3)+\mathcal O(t^{-\alpha/2})\mathcal O(|\tilde{\bf y}|^2)$ as $|\tilde{\bf y}|\to 0$ and $t\to\infty$. Hence, there exist $\tilde d_1\in (0,\tilde d)$ and $\tilde\tau_1\geq t_\ast$ such that $\ell_- |\tilde {\bf y}|^2\leq \tilde \ell(y_2,y_3,t)\leq \ell_+ |\tilde{\bf y}|^2$ for all $|\tilde{\bf y}|\leq \tilde d_1$ and $t\geq \tilde\tau_1$, where $\ell_-=\min\{|\nu|,|\lambda'(\psi_0)|\}/4$ and $\ell_+=\max\{|\nu|,|\lambda'(\psi_0)|\}$. The derivative of $\tilde \ell(y_2,y_3,t)$ along the trajectories of system \eqref{ysys1} satisfies 
\begin{align*}
\frac{d\tilde\ell}{dt} 
  = t^{-\alpha}\left( (|\nu|y_2^2+|\lambda'(\psi_0)|y_3^2)\frac{\vartheta(\psi_0)}{2}+\xi(y_1,y_2,y_3)+\mathcal O(|\tilde {\bf y}|^3)+\mathcal O(t^{-\frac{\alpha}{2}})\mathcal O(|\tilde{\bf y}|)\right)
\end{align*} 
as $ |\tilde{\bf y}|\to 0$ and $t\to\infty$ uniformly for all $|y_1|\leq \epsilon$, where
$\xi(y_1,y_2,y_3)\equiv |\nu|\lambda_{2,1}(\psi_0+y_3)y_1y_2-({\hbox{\rm sgn}}\,\nu) (\lambda(\psi_0+y_3)-\lambda(\psi_0))\nu_{2,1}y_1^2 $. It follows that there exist $\tilde d_2\in (0,\tilde d_1]$, $\tilde \tau_2\geq t_\ast$, and $\tilde B_1>0$ such that
$d\tilde \ell/dt\geq t^{-\alpha} (\vartheta(\psi_0)\ell_- |\tilde{\bf y}|^2-\tilde B_1(\epsilon+t^{-\frac{\alpha}{2}}) |\tilde {\bf y}|)$ for all $|y_1|\leq \epsilon$, $|\tilde {\bf y}|\leq \tilde d_2$ and $t\geq \tilde \tau_2$. Hence, for all $\tilde \delta\in(0,\tilde d_2/2)$ there exist 
\begin{gather*}
\epsilon=\min\left\{d_1,\frac{\tilde \delta \vartheta(\psi_0)\ell_-}{4\tilde B_1}\right\}, \quad 
\tau_\ast=\max\left\{\tau_\epsilon,\tilde\tau_1,\tilde \tau_2,\left(\frac{4\tilde B_1}{\tilde \delta \vartheta(\psi_0)\ell_-}\right)^{\frac{2}{\alpha}}\right\}
\end{gather*}
such that $d\tilde \ell/dt\geq t^{-\alpha} \vartheta(\psi_0)\ell_- |\tilde{\bf y}|^2/2\geq t^{-\alpha} \tilde c  \tilde \ell$ for all $|y_1|\leq \epsilon$, $\tilde \delta \leq |\tilde {\bf y}|\leq \tilde d_2$ and $t\geq \tau_\ast$ with $\tilde c=\vartheta(\psi_0) \ell_1/(2 \ell_+)>0$. Integrating this inequality as $t\geq t_s\geq \tau_\ast$ and taking $|\tilde {\bf y}(t_s)|=\tilde \delta$, we obtain
\begin{align*}
|{\bf y}(t)|&\geq  \tilde \delta \sqrt{\frac{\ell_-}{\ell_+}}\exp\left(-\frac{\tilde c}{2(1-\alpha)}\left(t^{1-\alpha}-t_s^{1-\alpha}\right)\right), & \alpha&<1, \\
|{\bf y}(t)|&\geq \tilde \delta \sqrt{\frac{\ell_-}{\ell_+}}\left(\frac{t}{t_s}\right)^{-\frac{\tilde c}{2}}, & \alpha&=1.
\end{align*}
We see that there exists $t_e>t_s$ such that $|\tilde {\bf y}(t_e)|\geq \tilde d_2/2$. Returning to the variables $(\hat\rho,\hat v,\hat\psi)$ leads to the desired result.
\end{proof}

\section{Analysis of the full system}
\label{sec6}

\begin{proof}[Proof of Theorem~\ref{Th2}]
The proof is divided into two parts. First, we establish the stability of the solution, and then we demonstrate its instability.

{\bf 1}. Let $\pi(\psi_0)<0$, $\lambda'(\psi_0)\nu<0$, $\vartheta(\psi_0)<0$, and let $(\hat\rho_0(t),\hat v_0(t),\hat\psi_0(t))$ be the solution of system \eqref{rvpmod}, described in Lemma~\ref{Lem1}. Substituting 
\begin{gather}\label{subs22}
\rho=\hat\rho_0(t)+y_1, \quad 
v=\hat v_0(t)+y_2, \quad 
\psi=\hat\psi_0(t)+y_3
\end{gather}
into \eqref{rvpp}, we obtain the following system for the new variables $({\bf y},\varphi)$:
\begin{gather}\label{ysys2}
\begin{split}
& \frac{d{\bf y}}{dt}={\bf G}({\bf y},t)+\tilde {\bf G}({\bf y},\varphi,t),\\ 
& \frac{d\varphi}{dt}=\omega_2(A)+\tilde {\mathcal G}(\hat\rho_0(t)+y_1,\hat v_0(t)+y_2,\hat\psi_0(t)+y_3,\varphi,t),
\end{split}
\end{gather}
where ${\bf y}=(y_1,y_2,y_3)^T$,  ${\bf G}({\bf y},t)$ is defined in \eqref{ysys0}, and 
\begin{gather*}
\tilde {\bf G}({\bf y},\varphi,t) \equiv \tilde {\bf a}(\hat\rho_0(t)+y_1,\hat v_0(t)+y_2,\hat\psi_0(t)+y_3,\varphi,t), \quad 
\tilde {\bf a}(\rho,v,\psi,\varphi,t)\equiv 
\begin{pmatrix} 
\tilde \Pi(\rho,v,\psi,\varphi,t) \\ 
\tilde \Lambda(\rho,v,\psi,\varphi,t)\\
\tilde \Omega(\rho,v,\psi,\varphi,t)
\end{pmatrix}.
\end{gather*}
From \eqref{tildeLLOP} it follows that $|\tilde {\bf G}({\bf y},\varphi,t)|=\mathcal O(t^{-3\alpha/2})$ as $t\to\infty$ uniformly for all $|{\bf y}|\leq \Delta$ and $\varphi\in\mathbb R$ with some $\Delta>0$, while ${\bf G}({\bf y},t)$ satisfies \eqref{Gest} and \eqref{varsigma}.
We choose the Lyapunov function in the form \eqref{Lfunct}. It can easily be checked that 
\begin{align*}
\frac{dL}{dt}& =\partial_t L({\bf y},t)+\left(({\bf G}({\bf y},t))^T+(\tilde{\bf G}({\bf y},\varphi,t))^T\right) \nabla_{\bf y} L({\bf y},t) \\
& = t^{-\alpha}\left(\pi(\psi_0)y_1^2+ (|\nu|y_2^2+|\lambda'(\psi_0)|y_3^2)\frac{\vartheta(\psi_0)}{2}+\mathcal O(|{\bf y}|^3)+\mathcal O(t^{-\varsigma})\mathcal O(|{\bf y}|^2)+\mathcal O(t^{-\frac{\alpha}{2}})\mathcal O(|{\bf y}|)\right)
\end{align*} 
as $|{\bf y}|\to 0$ and $t\to\infty$. Hence, there exist $\Delta_1\in (0,\Delta]$, $t_1\geq t_\ast$, and $C_1>0$ such that 
\begin{gather*}
L_- |{\bf y}|^2\leq L({\bf y},t)\leq L_+ |{\bf y}|^2, \quad 
\frac{dL}{dt}\leq  t^{-\alpha}\left(-\mu_\alpha |{\bf y}|^2+C_1 t^{-\frac{\alpha}{2}} |{\bf y}|\right)
\end{gather*}
for all $|{\bf y}|\leq \Delta_1$, $\varphi\in\mathbb R$ and $t\geq t_1$, where the parameters $\mu_\alpha>0$ and $L_\pm>0$  are defined in \eqref{mualpha} and \eqref{LLpm}. 
Therefore, for all $\epsilon\in(0,\Delta_1)$ there exist 
\begin{gather*}
\delta_\epsilon=\frac{2C_1 t_\epsilon^{-\frac{\alpha}{2}}}{\mu_\alpha}, \quad 
t_\epsilon=\max\left\{t_1,\left(\frac{4C_1}{\mu_\alpha\epsilon}\sqrt{\frac{L_+}{L_-}}\right)^{\frac{2}{\alpha}}\right\}
\end{gather*}
such that $dL/dt\leq -t^{-\alpha} \mu_\alpha |{\bf y}|^2/2$ for all $\delta_\epsilon\leq |{\bf y}|\leq \epsilon$, $\varphi\in\mathbb R$, and $t\geq t_\epsilon$. Combining this with the inequalities \eqref{supinf}, we see that the solution $({\bf y}(t),\varphi(t))$ of system \eqref{ysys2} with initial conditions $|{\bf y}(t_s)|\leq \delta_\epsilon$, $\varphi(t_s)\in\mathbb R$, and $t_s\geq t_\epsilon$ cannot leave the domain $\{({\bf y},\varphi)\in\mathbb R^4: |{\bf y}|\leq \epsilon\}$ for all $t\geq t_s$. Thus, taking into account \eqref{ch1}, \eqref{ch2}, and \eqref{subs22}, we obtain the desired result.

{\bf 2}. Consider the asymptotic solution $(\hat \rho_\ast(t),\hat v_\ast(t),\hat \psi_\ast(t))$ constructed in the proof of Lemma~\ref{Lem1}. Substituting
\begin{gather*}
\rho = \hat \rho_\ast(t)+y_1, \quad 
v = \hat v_\ast(t)+y_2, \quad 
\psi = \hat \psi_\ast(t)+y_3
\end{gather*}
into \eqref{rvpp}, we obtain 
\begin{gather}\label{ysys3}
\begin{split}
& \frac{d{\bf y}}{dt}={\bf G}_\ast({\bf y},t)+\tilde {\bf G}_\ast({\bf y},\varphi,t),\\ 
& \frac{d\varphi}{dt}=\omega_2(A)+\tilde {\mathcal G}(\hat\rho_\ast(t)+y_1,\hat v_\ast(t)+y_2,\hat\psi_\ast(t)+y_3,\varphi,t),
\end{split}
\end{gather}
where 
$\tilde {\bf G}_\ast({\bf y},\varphi,t) \equiv \tilde {\bf a}(\hat\rho_\ast(t)+y_1,\hat v_\ast(t)+y_2,\hat\psi_\ast(t)+y_3,\varphi,t)-{\bf N}_\ast(t)$, and the functions ${\bf G}_\ast({\bf y},t)$ and ${\bf N}_\ast(t)$ are defined in \eqref{Gast} and \eqref{Nast}. 

Let $\pi(\psi_0)>0$. Consider $\ell(y_1)=y_1^2/2$ as a Lyapunov function candidate for system \eqref{ysys3}. 
Note that $d\ell/dt=t^{-\alpha} \pi(\psi_0+y_3)y_1^2+\mathcal O(y_1)\mathcal O(t^{-3\alpha/2})$ as $y_1\to 0$ and $t\to\infty$ uniformly for all $\varphi\in\mathbb R$. 
There exists $\tilde d>0$ such that $\pi(\psi_0+y_3)\geq  |\pi(\psi_0)|/2$ for all $|y_3|\leq \tilde d$. 
It follows that there are $t_1\geq t_\ast$, $d_1>0$, and $B_1>0$ such that 
$d\ell/dt\geq t^{-\alpha}(|\pi(\psi_0)|y_1^2/2-B_1 t^{-\alpha} |y_1|)$ for all $|y_1|\leq d_1$, $|\tilde{\bf y}|:=\sqrt{y_2^2+y_3^2}\leq \tilde d$, $\varphi\in\mathbb R$, and $t\geq t_1$. 
Hence, for all $\delta\in (0,d_1/2)$ there exists $\tau_\ast=\max\{t_1,(4B_1/\pi(\psi_0)\delta)^{2/\alpha}\}$ such that 
$d\ell/dt\leq t^{-\alpha}(\pi(\psi_0)/2-B_1\tau_\ast^{-\alpha/2}\delta)y_1^2\geq t^{-\alpha} \pi(\psi_0) y_1^2/4=t^{-\alpha} c \ell$ for all $\delta\leq |y_1|\leq d_1$, $|\tilde{\bf y}|\leq \tilde d$, $\varphi\in\mathbb R$, and $t\geq \tau_\ast$ with $c=\pi(\psi_0)/2>0$. 
Integrating this inequality for $t\geq t_s\geq \tau_\ast$ with $|y_1(t_s)|=\delta$, $|\tilde{\bf y}(t_s)|\leq \tilde d$, and $\varphi(t_s)\in\mathbb R$, we see that there exists $t_e>t_s$ such that $|y_1(t_e)|\geq d_1/2$.

Now, consider the case $\lambda'(\psi_0)\nu>0$. For definiteness, let $\nu>0$. We take $\tilde \ell(y_2,y_3)=y_2 y_3$ as a Chetaev function for system \eqref{ysys3}. Note that 
$\tilde \ell(y_2,y_3)\geq 0$ for all $(y_2,y_3)\in \mathcal D:=\{y_2\geq 0, y_3\geq 0\}$ and $\tilde \ell (y_2,y_3)\leq |\tilde {\bf y}|^2/2$. The derivative of $\tilde \ell(y_2,y_3)$ satisfies 
$d\tilde\ell/dt=t^{-\alpha/2} (\nu y_2^2+\lambda'(\psi_0)y_3^2 +\mathcal O(|\tilde{\bf y}|^3)) + \mathcal O(t^{-\alpha})\mathcal O(|\tilde{\bf y}| )$ as $|\tilde {\bf y}|\to 0$ and $t\to\infty$ uniformly for all $|y_1|\leq d_1$ and $\varphi\in\mathbb R$. Hence, there exists $\tilde d_1>0$, $\tilde t_1\geq t_\ast$, and $\tilde B_1>0$ such that
$d\tilde \ell/dt\geq t^{-\alpha/2}( \tilde b |\tilde {\bf y}|^2 - \tilde B_1 t^{-\alpha/2} |\tilde {\bf y}| )$ for all $|y_1|\leq d_1$, $|\tilde {\bf y}|\leq \tilde d_1$, $\varphi\in\mathbb R$, and $t\geq \tilde t_1$ with $\tilde b=\min\{\nu,\lambda'(\psi_0)\}/2$. It follows that for all $\tilde\delta\in (0,\tilde d_1/2)$ there exists $\tilde \tau_\ast=\max\{\tilde t_1,(2\tilde B_1/\tilde b \tilde \delta)^{2/\alpha}\}$ such that 
$d\tilde \ell/dt\leq t^{-\alpha/2}(\tilde b-\tilde B_1\tilde \tau_\ast^{-\alpha/2}\delta)|\tilde {\bf y}|^2\geq t^{-\alpha/2} \tilde b |\tilde {\bf y}|^2/2\geq t^{-\alpha} \tilde b \tilde \ell$ for all $|y_1|\leq d_1$, $(y_2,y_3)\in\mathcal D$, $\tilde \delta \leq |\tilde{\bf y}|\leq \tilde d_1$, $\varphi\in\mathbb R$, and $t\geq \tilde \tau_\ast$. 
Integrating this inequality for $t\geq t_s\geq \tau_\ast$ with $|y_1(t_s)|\leq d_1$, $y_2(t_s)= \tilde \delta \cos u$, $y_3(t_s)= \tilde \delta \sin u$, $u\in (0,\pi/2)$, and $\varphi(t_s)\in\mathbb R$, we see that there exists $t_e>t_s$ such that $|\tilde {\bf y}(t_e)|\geq \tilde d_1/2$.

Finally, if $\pi(\psi_0)<0$, $\lambda'(\psi_0)\nu<0$, and $\vartheta(\psi_0)>0$, the proof of instability follows by repeating the arguments from the second part of Lemma~\ref{Lem1}.

Thus, returning to the original variables leads to the result of the theorem.
\end{proof}

\begin{proof}[Proof of Theorem~\ref{Th3}]
From the second equation in \eqref{rvpp} it follows that there exist $C_1>0$ and $t_1\geq t_\ast$ such that $|v(t)|\geq C_1 t^{1-\alpha/2}$ for $t\geq t_1$. Combining this with the third equation in \eqref{rvpp}, we see that there exist $C_2>0$ and $t_2\geq t_1$ such that $|\psi(t)|\geq C_2 t^{2-\alpha}$ for $t\geq t_2$. Hence, $|\psi(t)|\to \infty$ as $t\to\infty$, and $|v(t)|$ increases until it reaches the boundary of the domain $\mathfrak D_{\epsilon,t_\ast}$.
\end{proof}

\section{Examples}\label{sex}

\subsection{Example 1}
Consider again system \eqref{PS} with $\omega_1(E)\equiv 1/2+E$, $\omega_2(E)\equiv 1-E$, and perturbations \eqref{FGmod}. It can easily be checked that the application of Theorem~\ref{Th1} with $\kappa=\varkappa=1$ and $A=1/2$ yields system \eqref{rvpp} with $\nu=1$, $\lambda_{2,1}(\psi)\equiv 0$, $\lambda_{2,2}(\psi)\equiv 2\lambda(\psi)$,
\begin{gather*}
\pi(\psi)    \equiv \frac{1}{4}\left(a_{1,1}+a_{1,2}\cos\psi\right), \quad  
\lambda(\psi)   \equiv \frac{1}{4}
\left(a_{2,1}+a_{2,2}\cos\psi\right), \quad
 \nu_{2,0}(\psi)   \equiv - \frac{1}{2}\left(b_{1,2}+b_{2,1}\sin\psi\right).
\end{gather*}
If $|a_{2,1}/a_{2,2}|<1$, then assumption \eqref{asl} holds with
\begin{gather*}
\psi_0^\pm=\pm  \arccos\left(-\frac{a_{2,1}}{a_{2,2}}\right) +2\pi k, \quad k\in\mathbb Z.
\end{gather*}
Note that $\nu \lambda'(\psi_0^\pm)=-(a_{2,2}/4)\sin\psi_0^\pm$. Let $a_{2,2}>0$. Then $\nu \lambda'(\psi_0^-)>0$ and $\nu \lambda'(\psi_0^+)<0$. It follows from Lemma~\ref{Lem0} and Theorem~\ref{Th2} that the equilibrium $(0,0,\psi_0^-)$ of the truncated system \eqref{rvpls} and the corresponding regime in system \eqref{PS} are unstable. It can easily be checked that $\pi(\psi_0^+)=(a_{1,1}a_{2,2}-a_{1,2}a_{2,1})/(4 a_{2,2})$ and $\vartheta(\psi_0^+)=b_{2,1} a_{2,1}/(2 a_{2,2})$. If $a_{1,1}<a_{2,1}a_{2,1}/a_{2,2}$, the equilibrium $(0,0,\psi_0^+)$ is stable in the truncated system. If, in addition, $b_{2,1} a_{2,1}<0$, then it follows from Theorem~\ref{Th2} that the solution with $E_1(t)\approx 0$, $E_2(t)\approx A$, and $\varphi_1(t)-\varphi_2(t)\approx \psi_0^+$ is stable in system \eqref{PS} (see~Fig.~\ref{FigEx1PL}).

\begin{figure}
\centering
{
  \includegraphics[width=0.4\linewidth]{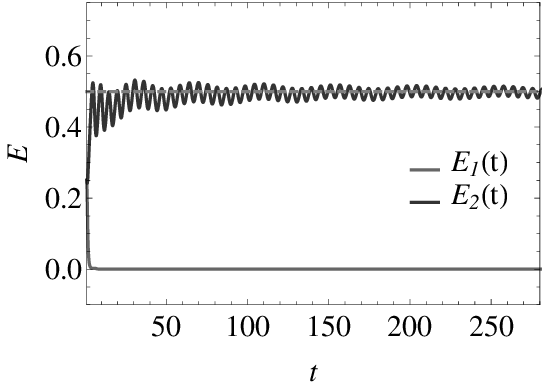}
}
\hspace{1ex}
{
   	\includegraphics[width=0.4\linewidth]{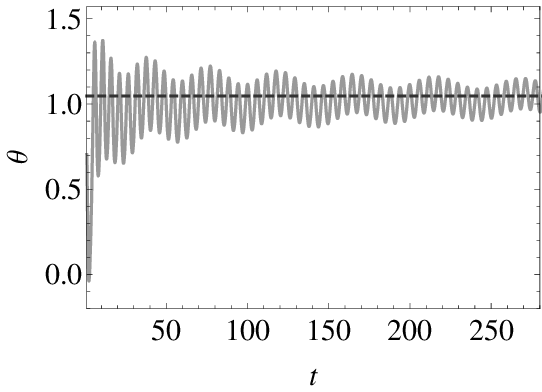}
}
\caption{\small The evolution of $E_1(t)$, $E_2(t)$, and $\theta(t)\equiv \varphi_1(t)-\varphi_2(t)$ for the solutions to system \eqref{PS} with $\omega_1(E)\equiv 1/2+E$, $\omega_2(E)\equiv 1-E$, and perturbations \eqref{FGmod}, where $a_{1,1}=b_{1,2}=-1$, $a_{1,2}=-2$, $b_{1,1}=a_{2,2}=b_{2,1}=b_{2,2}=1$, $a_{2,1}=-0.5$, and the initial conditions are $E_1(1)=E_2(1)=0.25$, $\varphi_1(1)=0.7$, $\varphi_2(1)=0$. The dashed curves correspond to $E(t)\equiv 0.5$ and $\theta(t)\equiv \pi/3$.} \label{FigEx1PL}
\end{figure}

If $|a_{2,1}/a_{2,2}|>1$, then assumption \eqref{asl2} holds, and it follows from Theorem~\ref{Th3} that the asymptotic regime with $E_1(t)\approx 0$, $E_2(t)\approx A$, and $\varphi_1(t)-\varphi_2(t)\approx {\hbox{\rm const}}$ does not arise (see~Fig.~\ref{FigEx1PD}).

\begin{figure}
\centering
{
  \includegraphics[width=0.4\linewidth]{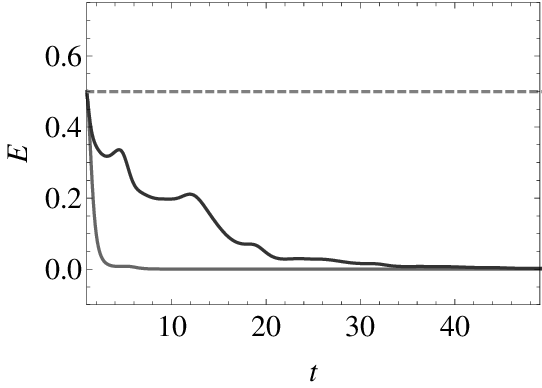}
}
\hspace{1ex}
{
   	\includegraphics[width=0.4\linewidth]{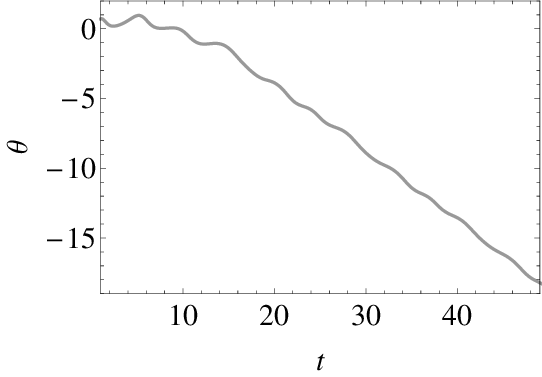}
}
\caption{\small The evolution of $E_1(t)$, $E_2(t)$, and $\theta(t)\equiv \varphi_1(t)-\varphi_2(t)$ for the solutions to system \eqref{PS} with $\omega_1(E)\equiv 1/2+E$, $\omega_2(E)\equiv 1-E$, and perturbations \eqref{FGmod}, where $a_{1,1}=b_{1,2}=-1$, $a_{1,2}=-2$, $b_{1,1}=a_{2,2}=b_{2,1}=b_{2,2}=1$, $a_{2,1}=-1.2$, and the initial conditions are $E_1(1)=E_2(1)=0.5$, $\varphi_1(1)=0.7$, $\varphi_2(1)=0$. The dashed curve corresponds to $E(t)\equiv 0.5$.} \label{FigEx1PD}
\end{figure}

\subsection{Example 2}
Consider system \eqref{Ex0} with $U_i(x)={w_i^2 x^2}/{2}+{u_i x^4}/{4}$, $u_1>0$, $u_2<0$, $w_i>0$, and $\alpha=1/2$. The corresponding limiting system $dx_i/dt=w_i y_i$, $dy_i/dt=-U'(x_i)/w_i$, $i\in\{1,2\}$ describes two different Duffing oscillators with stable equilibria at $(x_1,y_1)=(0,0)$ and $(x_2,y_2)=(0,0)$. The level lines $\{(x_i,y_i)\in\mathbb R^2: H_i(x_i,y_i)=E_i\}$ for $0<E_i<w_i^4/  |4u_i|$ correspond to periodic solutions with the frequencies (see~\cite[\S 4.2.1]{CL11})
\begin{align*}
\omega_1(E_1) & \equiv   \pi \sqrt{w_1^2+u_1 \rho_1^2} \left\{2K\left(\frac{u_1\rho_1^2}{2(w_1^2+u_1\rho_1^2)}\right)\right\}^{-1}, \quad 
\rho_1=\sqrt{\frac{\sqrt{w_1^2+4E_1 u_1}-w_1}{u_1}},\\
\omega_2(E_2) & \equiv  \pi \sqrt{w_2^2-\frac{|u_2| \rho_2^2}{2}} \left\{2K\left(\frac{|u_2|\rho_2^2}{2w_2^2-|u_2|\rho_2^2}\right)\right\}^{-1}, \quad 
\rho_2=\sqrt{\frac{w_2-\sqrt{w_2^2-4E_2 |u_2|}}{|u_2|}},
\end{align*}
where $K(k)$ is the complete elliptic integral of the first kind and $H_i(x,y)\equiv U_i(x)+w_i^2 y^2/2$. Note that $\omega_1(E_1)=w_1+3 E_1 u_1/(4w_1^3)+\mathcal O(E_1^2)$ as $E_1\to 0$.
Let $w_1=0.98$, $u_1=1/4$, $w_2=1$, and $u_2=-\varepsilon$ with $\varepsilon=0.1$. Then, condition \eqref{rc} holds with $\kappa=\varkappa=1$, $A\approx 0.257$, $\eta_1=3/ (16 w_1^3)>0$, and $\eta_2<0$ (see Fig.~\ref{FigRC}).
\begin{figure}
\centering
{
   \includegraphics[width=0.4\linewidth]{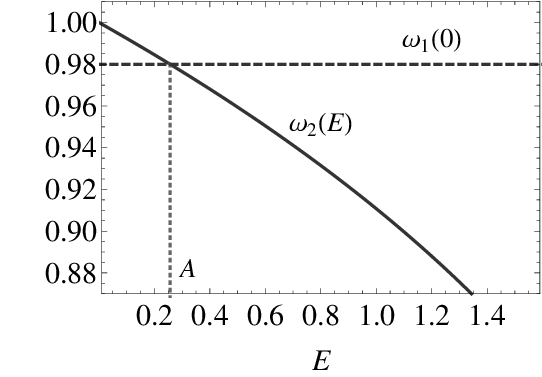}
}
\caption{\small Finding $A$, when $u_1=1/4$, $u_2=-0.1$, $w_1=0.98$, $w_2=1$, $\kappa=1$, and $\varkappa=2$.} \label{FigRC}
\end{figure}

1. Consider the perturbations in the form
\begin{gather}\label{G1G2}
\mathcal G_1(x_1,z_1,x_2,z_2)\equiv a_1 x_1 x_2^2 + b_1 z_1, \quad 
\mathcal G_2(x_1,z_1,x_2,z_2)\equiv \frac{x_1^2 (a_2 x_2 + b_2 z_2)}{w_1^2 x_1^2+z_1^2}
\end{gather}
with constant parameters $a_1,a_2,b_1$, and $b_2$. 
In this case, system \eqref{Ex0} describes two weakly coupled Duffing oscillators with time-decaying nonlinear parametric disturbances and phase-dependent coupling. Moreover, in the action-angle variables this system takes the form \eqref{PS}, where $f_1, g_1$, $f_2$, and $g_2$ are defined by \eqref{fgl}. It is not hard to check that the change of variables described in Theorem~\ref{Th1} transforms the system into \eqref{rvpp} with $\nu=-\eta_2>0$, $\lambda_{2,1}(\psi)\equiv 0 $,
\begin{align*}
\pi(\psi)& = \frac{1}{4w_1 }\left(2 w_1  b_1- A a_1 \sin 2\psi\right)+\mathcal O(\varepsilon), \\
\lambda(\psi)& =  \frac{A b_2}{4w_1^2 }\left(2 - c_2 \cos (2\psi+\theta_2)\right)+\mathcal O(\varepsilon),\\
\lambda_{2,2}(\psi)& =   \frac{1}{4w_1^2 b_2 }\left(2   - c_2 \cos(2\psi+\theta_2) \right)+\mathcal O(\varepsilon), \\
\nu_{2,0}(\psi)& = \frac{1}{8w_1^2}\left(-4 A a_1 w_1 + 2 a_2 - 2 A a_1 w_1 \cos 2 \psi + b_2 c_2 \sin(2\psi+\theta_2)\right)+\mathcal O(\varepsilon)
\end{align*}
as $\varepsilon\to0$, where $c_2=\sqrt{1+(a_2/b_2)^2}$ and $\cos \theta_2= 1/c_2$.

If $|c_2|>2$, then assumption \eqref{asl} holds with
\begin{gather*}
\psi_0^\pm=-\frac{\theta_2}{2} \pm \frac{1}{2} \arccos\left(\frac{2}{c_2}\right) +\pi k+\mathcal O(\varepsilon), \quad  \varepsilon \to 0, \quad k\in\mathbb Z.
\end{gather*}
If $\nu \lambda'(\psi_0)>0$, then it follows from Lemma~\ref{Lem0} and Theorem~\ref{Th2} that the equilibrium $(0,0,\psi_0)$ of the truncated system \eqref{rvpls} and the corresponding regime in system \eqref{Ex0} are unstable. Let $b_2<0$. Then $\nu \lambda'(\psi_0^+)<0$.
It is easily checked that $\pi(\psi_0^+)<0$ if and only if
\begin{gather*}
 \frac{A a_1 }{c_2^2}\left(2\sqrt{c_2^2-1}-\sqrt{c_2^2-4}\right) <-2 w_1 b_1.
\end{gather*} 
In this case, the equilibrium $(0,0,\psi_0^+)$ is stable in the truncated system. If, in addition, 
\begin{gather*}
\frac{A a_1   }{ c_2^2}\left(2\sqrt{c_2^2-1}-\sqrt{c_2^2-4}\right)> \frac{b_2}{w_1},
\end{gather*}
 then $\vartheta(\psi_0^+)<0$, and from Theorem~\ref{Th2} it follows that the solution with $E_1(t)\equiv H_1(x_1(t),y_1(t))\approx 0$, $E_2(t)\equiv H_2(x_2(t),y_2(t))\approx A$, and $\varphi_1(t)-\varphi_2(t)\approx \psi_0^+$ is stable in system \eqref{Ex0} (see~Fig.~\ref{FigE1PL}).

\begin{figure}
\centering
{
  \includegraphics[width=0.4\linewidth]{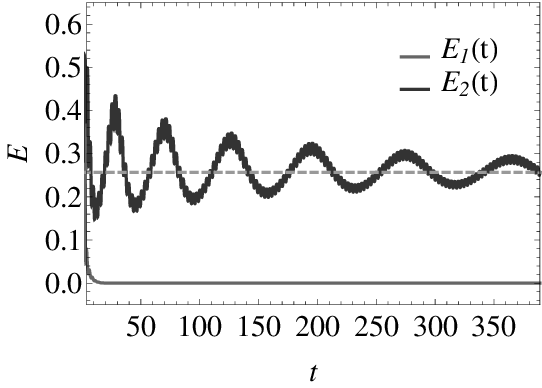}
}
\hspace{1ex}
{
   	\includegraphics[width=0.4\linewidth]{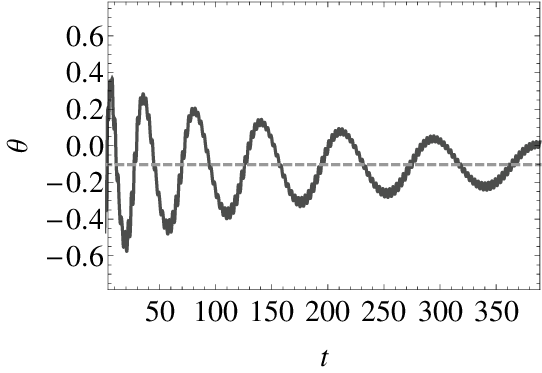}
}
\caption{\small The evolution of $E_1(t)$, $E_2(t)$, and $\theta(t)\equiv \varphi_1(t)-\varphi_2(t)$ for solutions to system \eqref{Ex0} with $U_i(x)\equiv w_i^2 x^2+u_i x^4/4$, right-hand sides \eqref{G1G2}, $u_1=1/4$, $u_2=-0.1$, $w_1=0.98$, $w_2=1$, $a_1=-5$, $b_1=-1$, $a_2=-2.5$, and $b_2=-0.5$. The dashed curves correspond to $E(t)\equiv 0.257$ and $\theta(t)\equiv \psi_0^+$, where $\psi_0^+\approx -0.1$.} \label{FigE1PL}
\end{figure}

If $|c_2|<2$, then assumption \eqref{asl2} holds, and it follows from Theorem~\ref{Th3} that the asymptotic regime with $E_1(t)\approx 0$, $E_2(t)\approx A$, and $\varphi_1(t)-\varphi_2(t)\approx {\hbox{\rm const}}$ does not arise (see~Fig.~\ref{FigE1PD}).

\begin{figure}
\centering
{
  \includegraphics[width=0.4\linewidth]{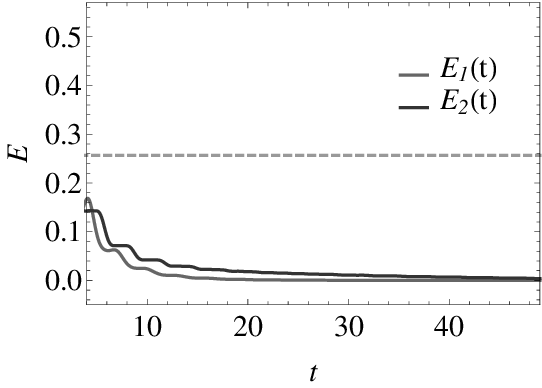}
}
\hspace{1ex}
{
   	\includegraphics[width=0.4\linewidth]{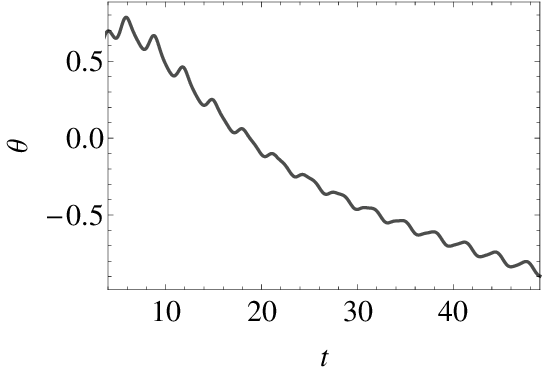}
}
\caption{\small The evolution of $E_1(t)$, $E_2(t)$, and $\theta(t)\equiv \varphi_1(t)-\varphi_2(t)$ for the solutions to system \eqref{Ex0} with $U_i(x)\equiv w_i^2 x^2+u_i x^4/4$, right-hand sides \eqref{G1G2}, $u_1=1/4$, $u_2=-0.1$, $w_1=0.98$, $w_2=1$, $a_1=-5$, $b_1=-1$, $a_2=-0.5$, and $b_2=-1$. The dashed curve corresponds to $E(t)\equiv 0.257$.} \label{FigE1PD}
\end{figure}

2. Consider now the perturbations in the form
\begin{gather*}
\begin{split}
\mathcal G_1(x_1,z_1,x_2,z_2)&\equiv x_1 (a_{1,0}+ a_{1,1} x_2^2+a_{1,2} z_2^2)+z_1 (b_{1,0}+ b_{1,1} x_2^2+b_{1,2} z_2^2), \\
\mathcal G_2(x_1,z_1,x_2,z_2)&\equiv x_2 (a_{2,0} + a_{2,1} x_1^2 + a_{2,2} z_1^2) + z_2 (b_{2,0} + b_{2,1} x_1^2 + b_{2,2} z_1^2)
\end{split}
\end{gather*}
with constant parameters $a_{i,j},b_{i,j}\in\mathbb R$. 
In this case, system \eqref{Ex0} corresponds to a pair of coupled Duffing oscillators with a nonlinear parametric coupling. The change of variables described in Theorem~\ref{Th1} transforms the system into \eqref{rvpp} with $\nu=-\eta_2>0$, $\lambda_{2,1}(\psi)\equiv 0 $,
\begin{align*}
\pi(\psi)& = \frac{b_{1,0}+A(b_{1,1}+  b_{1,2})}{2}-
\frac{A}{4 w_1}\left(w_1 (b_{1,1} -b_{1,2}) \cos 2 \psi  +  (a_{1,1}-a_{1,2}) \sin 2 \psi \right)+\mathcal O(\varepsilon), \\
\lambda(\psi)& =   A b_{2,0} +\frac{\varepsilon A^2 b_{2,0} }{4}+\mathcal O(\varepsilon^2),\\
\lambda_{2,2}(\psi)& =  b_{2,0}+\frac{\varepsilon A b_{2,0} }{2} +\mathcal O(\varepsilon^2), \\
\nu_{2,0}(\psi)& = \frac{w_1a_{2,0}-a_{1,0}- A (a_{1,1}+a_{1,2})}{2 w_1 }+  \frac{A}{4 w_1 }\left( w_1(b_{1,1}  - b_{1,2} )\sin 2 \psi- (a_{1,1}-  a_{1,2}) \cos 2 \psi \right)+\mathcal O(\varepsilon)
\end{align*}
as $\varepsilon\to 0$. If $b_2\neq 0$, then assumption \eqref{asl2} holds, and it follows from Theorem~\ref{Th3} that the asymptotic regime with resonant energy does not occur.

\section{Conclusion}

Thus, the influence of perturbations on the dynamics of a pair of coupled, non-identical oscillators has been investigated. In particular, the model asymptotically autonomous system \eqref{rvpmod} describing the averaged dynamics has been derived. It turns out that this system can be represented as a pendulum-type equation with decaying perturbations
\begin{gather}\label{rpsys}
\frac{d\hat\rho}{d\tau}=\mathfrak F\left(\hat\rho,\hat \psi,\frac{d\hat\psi}{d\tau},\tau\right), \quad 
\frac{d^2\hat \psi}{d\tau^2}-\nu \lambda(\hat \psi)=\mathfrak G\left(\hat \rho,\hat \psi,\frac{d\hat \psi}{d\tau},\tau\right), \quad 
\tau=\frac{2}{2-\alpha} t^{\frac{2-\alpha}{2}}
\end{gather}
where $\mathfrak F(\rho,\psi,u,\tau)$ and $\mathfrak G(\rho,\psi,u,\tau)$ are $2\pi$-periodic with respect to $\psi$ and polynomials in $u$ such that
\begin{gather*}
\mathfrak F(\rho,\psi,u,\tau) = \mathcal O\left(\tau^{-\frac{\alpha}{2-\alpha}}\right), \quad 
\mathfrak G(\rho,\psi,u,\tau) = \mathcal O\left(\tau^{-\frac{\alpha}{2-\alpha}}\right), \quad \tau\to\infty.
\end{gather*}  
The right-hand sides in \eqref{rpsys} depend on the coupling and perturbations of the original oscillatory systems. 
The study of the model system has revealed at least two distinct regimes: phase-locking, when the phases of the oscillators are synchronized, and phase drift, when the phase difference grows unboundedly. We have described conditions on the perturbations that guarantee the existence and stability of resonant solutions with energies close to the resonant values in the phase-locking regime. In contrast, stable resonant solutions are not observed in the phase-drift mode. 

It should be noted that, throughout this paper, we have assumed that the perturbations preserve the equilibrium of the first oscillator. If this condition does not hold, then the proposed theory cannot be applied, since the structure of the model equations and the corresponding phase-locking conditions are violated. This case will be considered separately. We also note that equation \eqref{rpsys} represents a nonautonomous perturbation of a system at a zero-Hopf equilibrium~\cite{GK07}, where the corresponding Jacobian matrix of the limiting system has one zero and a pair of purely imaginary eigenvalues. To the best of the author's knowledge, the influence of decaying perturbations and possible non-autonomous bifurcations in this case have not been studied in detail previously. This topic deserves special attention and will be discussed elsewhere.

\section*{Acknowledgments}
The research is supported by the Russian Science Foundation  (Grant No. 26-11-00021).

}
\end{document}